\documentclass[11pt,reqno]{amsart}

\usepackage[margin=3.5cm]{geometry}

\usepackage{url}
\usepackage{mathtools}
\usepackage[dvipsnames]{xcolor}

\usepackage{amsthm,amssymb,amsbsy,amsopn,amstext,cite,
                                               amsxtra,euscript,bm,mathrsfs}
\usepackage[
    colorlinks=true,
    linkcolor=MidnightBlue,
    citecolor=MidnightBlue,
    urlcolor=MidnightBlue
]{hyperref}

\hypersetup{breaklinks=true}

\let\underbrace\LaTeXunderbrace

\def\({\left(}
\def\){\right)}
\usepackage[english]{babel}

\newtheorem{thm}{Theorem}
\newtheorem{lem}[thm]{Lemma}

\newtheorem{cor}[thm]{Corollary}
\newtheorem{prop}[thm]{Proposition}

\theoremstyle{definition}
\newtheorem{definition}[thm]{Definition}

\theoremstyle{remark}
\newtheorem{remark}[thm]{Remark}
\newtheorem*{ack}{Acknowledgements}
\newtheorem*{AI}{AI statement}

\newcommand{\vol}{\operatorname{vol}}

\newcommand{\GL}{\operatorname{GL}}
\newcommand{\SL}{\operatorname{SL}}

\newcommand{\SO}{\operatorname{SO}}

\numberwithin{equation}{section}
\numberwithin{thm}{section}
\numberwithin{table}{section}

\def\cD{{\mathcal D}}

\def\cM{{\mathcal M}}

\def \Q {{\mathbb Q}}

\def \R {{\mathbb R}}
\def \Z {{\mathbb Z}}

\def\a{\mathbf{a}}
\def\b{\mathbf{b}}

\newcommand{\ZZ}{\mathbb Z}
\newcommand{\ZZp}{\mathbb Z_{\text{prim}}}
\newcommand{\RR}{\mathbb R}

\newcommand{\eps}{\varepsilon}

\newcommand{\adj}{\mathrm{adj}}

\newcommand{\z}{\mathbf{z}}
\newcommand{\y}{\mathbf{y}}

\newcommand{\x}{\mathbf{x}}
\DeclareMathOperator{\rank}{rank}

\begin{document}

\makeatletter
\def\subjclassname{$2020$ Mathematics Subject Classification}
\makeatother

\title{The divisor function for matrices}

\author[T. Browning]{Tim Browning}
\address{IST Austria\\
Am Campus 1\\
3400 Klosterneuburg\\
Austria}
\email{tdb@ist.ac.at, 
nikita.kalinin@ist.ac.at,
lena.wurzinger@ist.ac.at}

\author[N.  Kalinin]{Nikita P. Kalinin}

\author[A. Ostafe]{Alina Ostafe}
\address{School of Mathematics and Statistics, University of New South Wales, Sydney NSW 2052, Australia}
\email{alina.ostafe@unsw.edu.au}

\author[D. Schindler]{Damaris Schindler}
\address{University of G\"ottingen\\
Bunsenstraße 3--5\\
37073 G\"ottingen\\
Germany}
\email{damaris.schindler@mathematik.uni-goettingen.de}

\author[L. Wurzinger]{Lena Wurzinger}


\begin{abstract}  
We introduce a matrix divisor function 
$\tau_n(T,M)$, counting factorisations $AB=M$ for 
$n\times n$ integer matrices $A,B$ of height  at most $T$. For a fixed non-singular $M$, 
or for the zero matrix $M=O_n$, we prove 
an asymptotic formula for 
$\tau_n(T,M)$, as $T\to \infty$, using lattice point counting. We also prove an essentially sharp uniform   
upper bound for $\tau_n(T,M)$, for an arbitrary non-singular matrix $M$.
\end{abstract}

\maketitle

\thispagestyle{empty}
 \setcounter{tocdepth}{1}
 \tableofcontents

\section{Introduction}

For a positive integer $n$, 
let  $\cM_n\(\Z\)$   denote the set of $n\times n$ matrices with integer coefficients. 
For  $T\geq 1$, let $\cM_n\(\Z; T\)$   denote the
set  of 
$A = \(a_{ij}\)_{1\leq i,j\leq n}\in \cM_n(\Z)$ of 
height $|A|\leq T$, where $|A|=\max_{1\leq i,j\leq n}|a_{ij}|$.
In particular,   $\# \cM_n\(\Z; T\)$ has exact order $T^{n^2}$.
In this paper we consider a matrix analogue of the divisor function. Thus, for given $T\geq 1$ and 
$M\in \mathcal{M}_n(\ZZ)$, we let
\begin{equation}\label{eq:divisor}
\tau_n(T,M)=\#\left\{(A,B)\in \mathcal{M}_n(\ZZ)^2: |A|,|B|\leq T, ~AB=M
\right\}.
\end{equation}
We will be interested in the size of this function, both uniformly in $M$ (where upper  bounds are of primary concern) and for a fixed choice of  $M$ (where  an asymptotic formula is the goal).
 We will allow all implied constants to depend on $n$, but we shall strive for uniformity in $M$, where possible.
When $n=1$ it is natural to take $T=|M|$ and recover, up to a factor of 2, the classical divisor function (unless $M=0$).
This is no longer the case for $n\geq2$, however, since 
$\tau_n(T,M)$ can grow with $T$,  even when $M=I_n$ is the identity matrix. 
(An alternative notion of divisor function was studied by Bhowmik--Ramar\'e \cite{BR} and Bhowmik--Wu \cite{bhowmik_wu_zeta2001}. It is  defined as the number of left divisor classes $A \GL_n(\Z)$ such that $AB=M$, for some $B \in \mathcal{M}_n(\Z)$.)

We henceforth let 
\begin{equation}\label{eq:divisor-id}
\tau_n(T)=\tau_n(T,I_n),
\end{equation}
which counts the number of 
invertible matrices $A\in\cM_n\(\Z; T\)$ whose inverse is also in  $\cM_n\(\Z; T\)$.
In other words, $\tau_n(T)$ counts 
$A\in\cM_n\(\Z; T\)$ such that  $\det A=\pm 1$ and 
$|\Delta_{i,j}|\leq T$, where $\Delta_{i,j}$ ranges over  
all $(n-1)\times (n-1)$ minors, for $1\le i,j\le n$.
We will 
 use unipotent matrices to 
prove in Section~\ref{sec:lower} that 
\begin{equation}\label{eq:lower-easy}
\tau_n(T)\gg T^{\lfloor n^2/4\rfloor},
\end{equation}
for $n\geq 2$, which  illustrates growth in $T$. 
Note that we may write 
$
\tau_n(T)=2D_n(T),
$
where $D_n(T)=\#\cD_n(T)$ and
$$
\cD_n(T)=\{A \in \SL_n\(\Z; T\): |\Delta_{i,j}|\le T \text{ for $1\le i,j\le n$}\}.
$$
For $n=2$, we have $\cD_2(T)=\SL_2\(\Z; T\)$  and so it follows from a standard hyperbolic lattice point counting result (as in Krieg \cite{krieg}, for example) 
that 
\begin{equation}\label{eq:krieg}
\tau_2(T)=\frac{192}{\pi^2}T^2 +O(T \log T ).
\end{equation}
When  $n=3$, the quantity $D_3(T)$ has been studied by 
Maucourant
\cite{Mau}, 
with the outcome that an asymptotic formula is obtained for $D_3(T)$ 
 in \cite[Proposition 4.2]{Mau}.
(Note that Maucourant  works with the Euclidean norm, but the adjustment 
to the supremum norm is routine and merely changes the leading constant.)
 In private communication with the authors, Fran\c{c}ois Maucourant has explained how the main theorem in \cite{Mau} can be applied to 
 produce an asymptotic formula for  $D_n(T)$, for any $n\geq 3$.  
This yields 
 constants $c,\delta>0$ such that 
\begin{equation*}
\tau_n(T)=c T^{\lfloor n^2/2\rfloor} +O\left(T^{\lfloor n^2/2\rfloor-\delta}\right).
\end{equation*}
For comparison, when the cofactor size  constraint is dropped, 
 an asymptotic formula for  $\#\SL_n\(\Z; T\)$  is obtained by  Duke, Rudnick and Sarnak~\cite{DRS},  in which the main term has order $T^{n^2-n}$.

We now return to  the matrix divisor function 
$\tau_n(T,M)$ for an arbitrary non-singular 
matrix  $M\in \cM_n(\Z)$. 
For any full rank lattice $\mathsf\Lambda\subset \ZZ^n$ we let $B_{\mathsf\Lambda}$ denote a chosen basis matrix. 
We begin by recording an asymptotic formula, in which the implied constant depends implicitly on $M$.

\begin{thm}\label{thm:asymptotic_formula}
    Let $n \ge 2$ and let  $M \in \mathcal{M}_{n}(\Z)$ with $\det M \neq 0$. Then
    \begin{align*}
        \tau_n(T,M) \sim c_{M} T^{\lfloor n^2/2 \rfloor} ,
    \end{align*}
as $T\to \infty$, 
    where
    \[
    c_{M} = 2
    \sum_{
    \substack{
    M \Z^n \subseteq \mathsf{\Lambda} \subseteq \Z^n, \\
    \mathsf{\Lambda} \text{ lattice }
    }
    }
    \kappa (B_{\mathsf\Lambda}, B_{\mathsf\Lambda}^{-1}M),
    \] 
    for positive constants 
  $ \kappa (B_{\mathsf\Lambda}, B_{\mathsf\Lambda}^{-1}M)$ that will be defined in \eqref{def:kappa}.
\end{thm}

We shall deduce this result  from a very general lattice point counting result due to Gorodnik and Nevo \cite{gorodnik_nevo} in Sections \ref{section:latticepts} and \ref{sec:asymp}. 
    We expect it is possible to prove a version of Theorem~\ref{thm:asymptotic_formula} with an explicit error term. The main challenge lies in determining a fuller asymptotic expansion of the volume in Lemma~\ref{lem:volumecomputation}, which currently relies on the dominated convergence theorem.

If $\det(M) = \pm 1$, then the only lattice $\mathsf{\Lambda}$ with $M \Z^n \subseteq \mathsf{\Lambda} \subseteq \Z^n$ is $\mathsf{\Lambda} = \Z^n$. 
We therefore obtain the following consequence.

\begin{cor}\label{cor1.3}
    Let $n \ge 2$ and let  $M \in \mathcal{M}_{n}(\Z)$ with $\det M =\pm 1$. Then
    \begin{align*}
        \tau_n(T,M) \sim 
        2   \kappa (I_n, M)
        T^{\lfloor n^2/2 \rfloor} ,
    \end{align*}
    as $T\to \infty$, where  $\kappa(I_n,M)$ is the positive constant defined in \eqref{def:kappa}.
\end{cor}

In Section \ref{2x2} we shall check that this is consistent with \eqref{eq:krieg} in the special case $n=2$ and $M=I_2$.
Turning to uniform upper bounds, 
we are able to prove the following.

\begin{thm}\label{thm:uniform}
Let $n\geq 1$ and let $M\in \cM_n(\Z)$ with $\det M \neq 0$.  Then  
    \[
        \tau_n(T,M) \ll_{\varepsilon} T^{ \lfloor n^2/2 \rfloor +\varepsilon}, 
    \]
for any $\eps>0$, 
where the implied constant depends only on $\varepsilon$ and $n$.
\end{thm}

Without making further restrictions on $M$, we cannot hope for a 
bound which gets sharper as $|M|$ increases.  
Let  $n=2$ and let $M= m^2 I_2$ be a scalar matrix,  with $m$ a positive integer. Taking $T=2m$, it follows from Theorem \ref{thm:uniform} that 
$ \tau_2(T,m^2I_2) \ll_{\varepsilon} T^{ 2 +\varepsilon}$, where the implied constant only depends on $\varepsilon$. On the other hand, we notice that the equation $A A^\adj = m^2I_2$ holds whenever $\det A = m^2$. 
Thus $ \tau_2(T,m^2I_2)\gg T^{2-\varepsilon}$, since there  are $\gg m^{2-\varepsilon}$ solutions to the equation $ad-bc = m^2$,  with $a,b,c,d \in \ZZ\cap [-2m,2m]$, as follows from an application of the circle method 
\cite[Theorem~4]{HB}, for example. 
The authors are grateful to Victor Wang for this remark.

Our proof of Theorem \ref{thm:uniform} relies on  decomposing the set of all factorisations $AB =M$ according to the lattice $\mathsf\Lambda=A\Z^n$.
In Section~\ref{section:intermediate_lattices}, we give an upper bound for the number of different lattices $\mathsf\Lambda$ with $M\Z^n \subseteq \mathsf\Lambda \subseteq \Z^n$. Next, in  Section~\ref{section:boundingnumberofbases}, we prove an upper bound for the number of possible choices of $A$ with $A\Z^n = \mathsf\Lambda$, $|A|\le T$ and  $|A^{-1}M|\le T$,
where we recall that 
$|A|$ is used to denote the supremum norm of all the entries of a real $n\times n$ matrix $A$.
Finally, in Section~\ref{section:uniform_proof}, we combine these two bounds to prove Theorem~\ref{thm:uniform}. 

\medskip

We proceed to record some estimates for  
$\tau_n(T,M)$ when  $M\in \mathcal{M}_n(\ZZ)$ is  singular. The following lower bound  will be deduced from Theorem \ref{thm:asymptotic_formula} in Section \ref{sec:lower}.

  \begin{thm}\label{thm:lowerboundrankr}
  Let $M \in \mathcal{M}_n(\Z)$ have  rank $r$. Then
    \[
    \tau_{n}(T,M) \gg_M T^{\lfloor r^2/2 \rfloor + n(n-r)}.
    \]
    \end{thm}

This result shows that the restriction 
to non-singular matrices $M$
 is essential in Theorem \ref{thm:uniform}. In the special case $M=O_n$, 
where $O_n$ is the $n\times n$ zero matrix, we are also able to prove 
the following asymptotic formula  in Section \ref{sec:On}.

\begin{thm}\label{thm:factoring_zero}
    Let $n \ge 2$. Then
    \begin{align*}
       \tau_n(T,O_n) = c_{O_n} T^{n^2} +O(T^{n^2-1}),
    \end{align*}
    where $c_{O_n}$
    is a positive constant given in \eqref{c_n}.
\end{thm}

Note that this is consistent with the lower bound in  Theorem \ref{thm:lowerboundrankr} when $r=0$.
The  proof 
of Theorem~\ref{thm:factoring_zero} will be adapted from arguments  of Katznelson \cite{katznelson} on counting singular matrices in $\mathcal{M}_n(\Z)$.

\medskip

Our investigation of  the matrix divisor function $\tau_n(T,M)$ is  inspired by recent work  of 
Afifurrahman~\cite{Afif}, who was the first  to consider counting matrices  $A_1,\ldots,A_m\in\cM_n(\Z;T)$ satisfying a multiplicative equation
$
A_1\cdots A_m=M,
$
for some fixed integer matrix $M$.  Taking $m=2$ and $n\geq 2$ in his work, it  follows from \cite[Theorem~2.1]{Afif} that 
$$
\tau_n(T,M)\ll_\eps
\begin{cases}
T^{n^2-n+\eps} & \text{ if $\det M\neq 0$,}\\
T^{n^2+\eps} & \text{ if $\det M= 0$,}
\end{cases}
$$
for any $\eps>0$. The second estimate is essentially tight, as follows from 
Theorem \ref{thm:factoring_zero}. The first estimate is tight when $n=2$, but worse than our Theorem \ref{thm:uniform} for all $n\geq 3$.

\begin{remark}
Given $M\in \mathcal{M}_n(\ZZ)$ it is also natural to study a version of 
\eqref{eq:divisor} in which non-square matrices are allowed. 
Letting $\mathcal{M}_{a\times b}(\ZZ)$ denote the set of $a\times b$ matrices with integer coefficients, this would involve the quantity
$$
\tau_{n,m}(T,M)=\#\left\{(A,B)\in \mathcal{M}_{n\times m}(\ZZ)\times 
\mathcal{M}_{m\times n}(\ZZ): |A|,|B|\leq T, ~AB=M\right\},
$$
for $T\geq 1$.  It would be interesting to see what analogues of Theorems \ref{thm:asymptotic_formula}, \ref{thm:uniform} 
or \ref{thm:factoring_zero}
are possible in this setting. 
\end{remark}

\begin{remark}\label{rem:nonabeliancircle}
In recent work of Arala et al.\ \cite{arala}, an asymptotic formula is obtained for the number of 
zeros of the Diophantine equation
$$
 \gamma_1^2\pm \cdots \pm \gamma_m^2=0,
$$
for $m\geq 9$ and an arbitrary choice of signs, where $\gamma_1,\dots,\gamma_m$ run through Hurwitz quaternions of bounded norm.  It would be desirable to obtain similar results when the variables run over elements of $\mathcal{M}_n(\ZZ;T)$. Following a strategy of Birch--Davenport--Lewis \cite{BDL}, it is plausible that strong uniform upper bounds for $\tau_n(T,M)$ could help with 
understanding the density of solutions to equations of the form 
$$
A_1B_1 + \dots + A_mB_m = 0,
$$
where $A_i,B_i$ run over elements of $\mathcal{M}_n(\ZZ;T)$.
There is a direct connection with Remark~\ref{rem:nonabeliancircle}.  Indeed, on putting
\[
 \mathcal A=(A_1\ \cdots\ A_m)\in M_{n\times mn}(\mathbb Z),
 \qquad
 \mathcal B=
 \begin{pmatrix}
 B_1\\ \vdots\\ B_m
 \end{pmatrix}
 \in M_{mn\times n}(\mathbb Z),
\]
we have
$ \mathcal A\mathcal B=A_1B_1+\cdots+A_mB_m$, so that  the number of solutions under consideration is precisely
$\tau_{n,mn}(T,O_n)$.  Moreover, a direct application of work by Rydin Myerson
\cite[Theorem~1.2]{simon}  on systems of quadratic forms, 
gives an asymptotic formula 
for this counting problem whenever $m>4n$.  
It would be very interesting to determine whether 
the range $m>4n$ is improvable by working directly with matrices. 
\end{remark}

\begin{ack}
The authors are very grateful to 
Igor Shparlinski and Victor Wang for useful comments on an earlier draft, and to 
Fran\c{c}ois
 Maucourant  for his explanations about applying his work 
\cite{Mau} to the situation where $n\geq 4$.
During the preparation of this paper,  A.O. was supported by the Australian Research Council Grant FT250100208.
\end{ack}

\begin{AI}
Generative AI tools were used to suggest proof outlines for
Theorems~\ref{thm:asymptotic_formula} and~\ref{thm:uniform}.
Specifically, Claude Opus 4.8 and GPT-5.5 Pro were used in connection
with Theorem~\ref{thm:asymptotic_formula}, and GPT-5.6 Sol Pro was
used in connection with Theorem~\ref{thm:uniform}. All arguments
were independently checked by the authors, who take full
responsibility for the contents of the paper. The final text was
written by the authors.
\end{AI}

\section{Preliminaries} 

\subsection{Notation and background }
The height $|A|$ of a  matrix $A \in \cM_n(\R)$ is 
the supremum norm of all its  entries.
For any matrices $M,L \in \mathcal{M}_n(\RR)$, we recall the basic inequalities
\begin{align}\label{eq:basic_infty_norm_inequ}
    | M + L | \le  | M| +
    | L| \quad
    \text{ and } \quad
    | M  L| \le  n| M | 
    | L|.
\end{align}
The  second inequality also gives $(n|L^{-1}|)^{-1}
    | M | 
    \le 
    | M  L| $ for invertible $L$.

We denote by $\vol$ the Haar measure on 
$\SL_n(\R)$, normalised in such a way that 
\begin{align}\label{eq:siegelnormalisation}
    \vol (\SL_n(\R) / \SL_n(\Z)) = 1.
\end{align}
This measure may be obtained from the Lebesgue measure $\mu$ on $\mathcal{M}_{n}(\R) \cong \R^{n^2}$ by setting 
\[
    \vol(S) =
    \frac{
    \mu 
    \left( 
    \{
    \lambda M : M \in S, \; \lambda \in (0,1) 
    \}
    \right)
    }{\frac{1}{n} \prod_{j=2}^n \zeta (j)},
\]
whenever $\{
    \lambda M: M \in S, \; \lambda \in (0,1) 
    \}$ is Lebesgue-measurable. This, up to normalisation, is Siegel's construction \cite[\S3]{siegemvt}; a  contemporary exposition is provided in \cite[\S2]{volumesgilletgrayson}. 
    Denoting by $\mathcal{F}$ any measurable fundamental region for the action of $\SL_n(\Z)$ on $\SL_n(\R)$, Siegel's formula \cite[Lecture XV]{siegel} gives
$$
    \mu \left( 
    \{
    \lambda F : F \in \mathcal{F}, \; \lambda \in (0,1) 
    \}
    \right) = 
    \frac{1}{n} \prod_{j=2}^n \zeta (j),
$$
which yields \eqref{eq:siegelnormalisation}. Concretely, Siegel computes the volume of the set 
$$
\{M \in \mathcal{A} : |\det(M)| \le 1 \}, 
$$
where $\mathcal{A}$ is a measurable fundamental domain for the action of $\GL_n(\Z)$ on the set of invertible matrices $\{M \in \mathcal{M}_n(\R): \det M \neq 0 \}$. As $\GL_n(\Z)$ contains matrices $U$ of determinant $\det(U)=-1$, which can flip determinant signs, one may choose a fundamental domain where all matrices are of positive determinant; here, two matrices are equivalent under the action of $\GL_n(\Z)$ if and only if they are already equivalent under the action of $\SL_n(\Z)$. 
We observe that such a fundamental region $\mathcal{A}$ is given by $\{\lambda F : F \in \mathcal{F}, \: \lambda > 0 \}$.

 We recall that every $g \in \SL_n(\R)$ has a Cartan decomposition as
    $g = k_1 a_{\mathbf{t}} k_2$, where
    \begin{align}
    \label{eq:cartan_decomp}
         k_1, k_2 \in \SO_n(\R) =:K, \quad a_{\mathbf{t}} = \mathrm{diag}(e^{t_1}, \ldots, e^{t_n}), \quad \mathbf{t} = (t_1,\ldots, t_n) \in H^{+}, 
    \end{align}
    with     \begin{align}\label{eq:defH+}
  H^{+} := 
 \left \{
  \mathbf{t} \in \R^n : t_1 \ge \cdots \ge t_n , \ \sum_{i=1}^n t_i = 0
 \right  \} \ \subseteq \
 \left \{ 
         \mathbf{t} \in \R^n : \sum_{i=1}^n t_i = 0
         \right \} =: H .
     \end{align} 
     The matrices $k_1, k_2$ are not uniquely determined by $g$, whereas the matrix $a_{\mathbf{t}}$, a
  diagonal matrix with decreasing positive entries along the diagonal and determinant 1, is uniquely determined by $g$. We can interpret the  Haar measure on $\SL_n(\R)$ as 
  the pushforward under the map $K \times H^{+} \times K \to \SL_n(\R),  \ (k_1, \mathbf{t}, k_2) \mapsto k_1 a_{\mathbf{t}} k_2$ of the measure 
  \begin{align}\label{eq:cartanvol}
  c_n \prod_{i<j} \sinh(t_i - t_j) \;
      \mathrm{d}k_1
      \mathrm{d} \mathbf{t}
      \mathrm{d}k_2
      , \qquad
      c_n =  
      \frac{
      2^{\binom{n}{2}+n-1} \pi^{\frac{n(n+1)}{2}-1}
      }{ \prod_{j=2}^n
       (\Gamma(\tfrac{j}{2})^{2} \zeta(j))
      },
  \end{align} 
   where $\mathrm{d}k_1$, $\mathrm{d}k_2$ are Haar probability measures on $K$, $\mathrm{d}\mathbf{t}$ is given by the push-forward of the Lebesgue measure on $\R^{n-1}$ under the parametrisation $-t_n = t_1 + \cdots + t_{n-1}$, and $c_n$ is the normalisation constant consistent with \eqref{eq:siegelnormalisation}. 
   All of this is explained in Helgason \cite[Theorem~5.8]{helgasongroupsandgeom} for the Cartan decomposition of more general $G$, 
   or by Forrester
   \cite[Sections~2.1--2.2, especially (2.6), (2.8) and (2.14)]{cartandecomp}
   for the specialisation to $G = \SL_n(\R)$ and the value of $c_n$.

For invertible matrices $X, Y \in \mathcal{M}_{n}(\Z)$, we set
\begin{align}\label{def:kappa}
    \kappa(X,Y) = \lim_{T \to \infty} 
   T^{-\lfloor n^2/2 \rfloor} \vol (E_{X,Y}(T)),
\end{align}
where
\begin{align}\label{eq:defxy}
    E_{X,Y}(T) = 
    \{
    g \in \SL_n(\R) :  | X g | \le T , | g^{-1} Y| \le T
    \}.
\end{align}
In Lemma \ref{lem:volumecomputation} we shall prove that  $\kappa(X,Y)$ exists and is a positive constant.

\subsection{Lower bounds}\label{sec:lower}

We begin by proving the easy lower bound \eqref{eq:lower-easy},
before 
turning to the proof  of Theorem~\ref{thm:lowerboundrankr}.
Let  $n\geq 2$ and 
write $n=r+s$, with $r=\lfloor n/2\rfloor$ and $s=\lceil n/2\rceil$. 
For any $r\times s$ matrix $X$ with integer coefficients
and height $|X|\leq T$, consider the matrix
$$
A_X=
\begin{pmatrix} I_r & X\\ 0 & I_s
\end{pmatrix}.
$$
Then clearly $A_X\in \SL_n(\mathbb Z)$ and
$$
A_X^{-1}=
\begin{pmatrix} I_r & -X\\ 0 & I_s
\end{pmatrix}.
$$
Thus $A_X$ and $A_X^{-1}$ both have height at most $T$, so that
$$
\tau_n(T)\ge (2\lfloor T\rfloor+1)^{rs}\gg T^{\lfloor n^2/4\rfloor},
$$
as claimed in \eqref{eq:lower-easy}.

\begin{proof}[Proof of Theorem~\ref{thm:lowerboundrankr}]
Let $M \in \mathcal{M}_n(\Z)$ be a matrix of rank $r$. 
If $r=0$ take $A=O_n$ and $B$ arbitrary to get the lower bound 
$\tau_n(T,M)\gg T^{n^2}$. 
Suppose now that $r\geq 1$.
Then, by the existence of the Smith normal form, there exist invertible matrices $U,V \in \GL_n(\Z)$ such that $UMV$ is a diagonal matrix of the form
\[
    U M V = \mathrm{diag}(d_1, \ldots, d_r, 0, \ldots, 0),
\]
where $d_1, \ldots, d_r \neq 0$. Every factorisation $UMV = AB$ provides a factorisation of the form $M = (U^{-1}A)(BV^{-1})$ and vice versa. Moreover, as $M$ is fixed, so are $U$ and $V$. 
Thus  $|A| \ll |U^{-1}A| \ll |A|$ and $|B| \ll |BV^{-1}| \ll |B|$. It follows that there exists a constant $k>0$ depending on $M$ such that 
\begin{equation}\label{eq:comparisontaum_smithnormalform}
       \tau_n(T,M) \ge \tau_n(kT,UMV).
\end{equation}

Let us assume that there are integer matrices $L, R\in \mathcal{M}_r(\Z)$ such that $LR = \mathrm{diag}(d_1,\ldots,d_r)$. Then, for any $(n-r)\times r$ matrix $X_1$ and any $(n-r)\times (n-r)$ matrix $X_2$, we have 
\[
    \begin{pmatrix}
        L & O_{r,n-r} \\
        O_{n-r,r} & O_{n-r}
    \end{pmatrix}
    \begin{pmatrix}
        R & O_{r,n-r} \\
        X_1 & X_2
    \end{pmatrix} 
    = \mathrm{diag}(d_1,\ldots,d_r,0,\ldots,0).
\]
Here $O_{i,j}$ denotes the $i\times j$ zero matrix, and we write $O_{i,i}=O_i$.
If $r=1$, then there are $\gg 1 = T^{\lfloor r^2 /2 \rfloor}$ choices for $(L,R)$, since one may for instance take $(L,R) = (d_1,1)$ for sufficiently large $T \ge |d_1|$; if $r\ge 2$, then there are $\gg T^{\lfloor r^2/2 \rfloor}$ choices for $(L,R)$ by Theorem~\ref{thm:asymptotic_formula}. Moreover, there are $\gg T^{n(n-r)}$ choices for $X_1, X_2$. 
Therefore, there are 
\[
   \tau_{n}(T,UMV) \gg T^{\lfloor r^2/2 \rfloor + n(n-r)}
\]
 factorisations $A'B' = UMV$, with $|A'|, |B'| \le T$.
The claim now follows by \eqref{eq:comparisontaum_smithnormalform}.
\end{proof}

\section{Lattice point counting on
\texorpdfstring{$\SL_n(\R)$}{SL\_n(R)}}\label{section:latticepts}

 Recall that a group $G$, equipped with a topology $\mathcal{T}$,
is called {\it second countable} if $\mathcal{T}$ has a countable basis, and that it is called
{\it locally compact} if $\mathcal{T}$ is Hausdorff and every $g \in G$ has a compact neighbourhood. 
The main tool for the proof of Theorem~\ref{thm:asymptotic_formula} is a lattice point counting result due to Gorodnik and Nevo \cite{gorodnik_nevo} for
locally compact second countable groups.
Throughout this section, we fix a locally compact second countable group $G$ with a Haar measure $m_G$ and a lattice $\Gamma \subset G$. Furthermore, we let  $(B_T)_{T>0}$ be a family of bounded Borel subsets of $G$ of positive finite Haar measure and write 
$$
    \beta_T = m_G(B_T)^{-1} \mathbf{1}_{B_T} \mathrm{d}m_G
$$
for the normalised Haar measure on $B_T$. We let $(\mathcal{O}_\varepsilon)_{\varepsilon > 0}$ be a symmetric family of neighbourhoods of the identity in $G$, meaning that $g\in \mathcal{O}_\varepsilon \Leftrightarrow g^{-1}\in \mathcal{O}_\varepsilon$, 
 with the properties that 
\begin{itemize}
\item[(i)] $\mathcal{O}_{\varepsilon_1} \subset \mathcal{O}_{\varepsilon_2}$ if $\varepsilon_1 < \varepsilon_2$;
\item[(ii)] 
the projection $\mathcal{O}_{\varepsilon}  \mathcal{O}_{\varepsilon} \to G / \Gamma, g \mapsto [g]_{G/ \Gamma}$, is injective for all sufficiently small $\varepsilon$, where
\begin{equation}\label{eq:OO}
\mathcal{O}_{\varepsilon}  \mathcal{O}_{\varepsilon} = \{g_1 g_2: g_1, g_2 \in \mathcal{O}_\varepsilon \},
\end{equation}
 and $[g]_{G/ \Gamma}$ is the equivalence class of $g \in G$ under right-multiplication by $\Gamma$.
\end{itemize}

Our application concerns the locally compact second countable group $\SL_n(\R)$ with the lattice $\SL_n(\Z)$. The rôle of the family of sets $(B_T)_{T>0}$ will be played by the family $(E_{X,Y}(T))_{T>0}$  defined in \eqref{eq:defxy}; we introduce our specific choice of $(\mathcal{O}_\varepsilon)_{\varepsilon>0}$ below in \eqref{eq:nbh_id}.  The main goal of this section is to derive the following result from 
Gorodnik--Nevo \cite{gorodnik_nevo}.

\begin{thm} \label{lem:maincountinglemma}
    Let $n\ge 2$, let  $X, Y \in \mathcal{M}_{n}(\Z)$ be such that $\det X\neq 0$ and $\det Y\neq 0$,  and let the sets $(E_{X,Y}(T))_{T>0}$ be  defined as in \eqref{eq:defxy}. Then 
    \[
    \# (E_{X,Y}(T) \cap \SL_n(\Z)) \sim \kappa(X,Y) T^{\lfloor
    n^2/2 \rfloor} 
    \]
    as $T \to \infty$.
\end{thm}

We now introduce all the definitions necessary to state the main theorem due to Gorodnik and Nevo, on which the proof of Theorem~\ref{lem:maincountinglemma}
 is founded.

\begin{definition}[{Well-roundedness \cite[Definition~1.1]{gorodnik_nevo}}] \label{def:wellrounded}
    For
 any $\varepsilon > 0$, we set
\begin{align*}
    B^{+}_T (\varepsilon)= \mathcal{O}_\varepsilon B_T \mathcal{O}_\varepsilon \quad \text{ and } 
    \quad 
    B^{-}_T (\varepsilon) =\bigcap_{h_1, h_2 \in \mathcal{O}_\varepsilon} h_1 B_T h_2.
\end{align*}
We say that $(B_T)_{T>0}$ is {\it well-rounded} with respect to $(\mathcal{O}_\varepsilon)_{\varepsilon > 0}$ if for every $\delta >0$, there exist $\varepsilon_0 >0$, $ T_0>0$ such that for all $T \ge T_0$ and $ \varepsilon \in (0, \varepsilon_0)$, we have 
\begin{align*}
    m_G(B^{+}_T (\varepsilon)) \le (1 + \delta) m_G(B^{-}_T (\varepsilon)).
\end{align*}

\end{definition}

\begin{definition}[{Mean ergodic theorem \cite[Definition~1.4]{gorodnik_nevo}}] \label{def:meanergodictheorem}
       On the quotient $X = G/\Gamma$, let $\mu$ denote the  probability measure  obtained by normalising the measure on $X$ induced by $m_G$. For $f \in L^2(X)$, we define the averaging operator $\pi_X(\beta_T)$ by 
       \[ \pi_X(\beta_T)(f): x \longmapsto 
       \frac{1}{m_G(B_T)}
       \int_{B_T}
       f(g^{-1}x)\mathrm{d}m_G(g).
       \]
        The operators $(\pi_X(\beta_T))_{T>0}$
       satisfy the {\it mean ergodic theorem in $L^2(X)$} 
       if
       \begin{align*}
           \left\|
           \pi_X(\beta_T)(f) -  
           \int_X
           f \mathrm{d} \mu
           \right\|_{L^2(X)} \longrightarrow 0 \quad \text{ as } \quad T \to \infty
       \end{align*}
       for all $f \in L^2(X)$. The ergodic theorem is called \textit{stable} if there is some $\varepsilon_0>0$ such that it holds for all families $B_T^+(\varepsilon)$ and $B_T^-(\varepsilon)$, $\varepsilon \in (0,\varepsilon_0)$, simultaneously. 
\end{definition}

With the terminology in place, we now state the quantitative lattice point counting theorem of Gorodnik and Nevo \cite[Theorem~1.7]{gorodnik_nevo}.

\begin{thm}[Gorodnik--Nevo] \label{thm:gorodnik_nevo}
  Assume that the family $(B_T)_{T>0}$ is well-rounded with respect to  $(\mathcal{O}_\varepsilon)_{\varepsilon>0}$ and that the averages along $(B_T)_{T>0}$  satisfy the stable mean ergodic theorem in $L^2(G/\Gamma)$. Then, as $T \to \infty$, we have   
    \[
    \# (\Gamma \cap B_T) \sim \frac{m_G(B_T)}{m_G( G/\Gamma )}.
    \] 
\end{thm}

To prove Theorem~\ref{lem:maincountinglemma}, we verify the hypotheses of Theorem~\ref{thm:gorodnik_nevo} for the family $(E_{X,Y}(T))_{T>0}$ and establish an asymptotic formula for $\vol(E_{X,Y}(T))$.

\subsection{The volume of $E_{X,Y}(T)$} \label{section:volume}
   In this section, we compute an asymptotic formula for the volumes of the sets $E_{X,Y}(T)$ as $T \to \infty$. 

\begin{lem}\label{lem:volumecomputation}
    For fixed invertible matrices $X,Y \in \mathcal{M}_{n}(\Z)$, the limit $\kappa (X,Y)$ in
    \eqref{def:kappa} exists and satisfies $\kappa (X,Y)>0$.
\end{lem}
\begin{proof}
We use the 
Cartan decomposition     $g = k_1 a_{\mathbf{t}} k_2$ of any 
 $g \in \SL_n(\R)$, in the notation of \eqref{eq:cartan_decomp}.
   Set $T = e^L$ and write
  \begin{align*} 
  f_L(k_1,\mathbf{t},k_2) &= \mathbf{1}_{[| X k_1 a_{\mathbf{t}} k_2 | \le T]} 
      \mathbf{1}_{[| (k_1 a_{\mathbf{t}} k_2)^{-1} Y | \le T]}  \\
      &=
      \mathbf{1}_{[| e^{-L} X k_1 a_{\mathbf{t}} k_2 | \le 1]}  
      \mathbf{1}_{[|e^{-L} (k_1 a_{\mathbf{t}} k_2)^{-1} Y | \le 1]} .
  \end{align*}
  By \eqref{eq:cartanvol}, we have  
  \begin{align*}
      \vol(E_{X,Y}(T)) = c_n \int_K \int_{H^{+}}  \int_K f_L(k_1,\mathbf{t},k_2) \prod_{i<j} \sinh(t_i - t_j) \;
      \mathrm{d}k_1
      \mathrm{d} \mathbf{t}
      \mathrm{d}k_2.
  \end{align*}
  To this integral we apply the change of variables $\mathbf{x} = \mathbf{t} - L\mathbf{v}$, where we set 
  \begin{align}\label{eq:definitionof v}
  \mathbf{v} =
\begin{cases}
(\underbrace{1,\dots,1}_{n/2},
 \underbrace{-1,\dots,-1}_{n/2}) &\text{if $n$ is even},\\
(\underbrace{1,\dots,1}_{(n-1)/2},
 0,
 \underbrace{-1,\dots,-1}_{(n-1)/2}) & \text{if $n$ is odd.}
\end{cases}
  \end{align}
  We further set  $\mathcal{D}_L = H^+ -L \mathbf{v}$. It follows that the volume of $E_{X,Y}(T)$ may be computed as
  \begin{align}\label{eq:volEafterchangeofvariables}
      c_n \int_K \int_{H}  \int_K 
     \mathbf{1}_{\mathcal{D}_L}(\mathbf{x})
     f_L(k_1,\mathbf{x}+L\mathbf{v},k_2) \prod_{i<j} \sinh(x_i - x_j + L(v_i-v_j)) \;
      \mathrm{d}k_1
      \mathrm{d} \x
      \mathrm{d}k_2,
  \end{align}
  where $H$ is given in \eqref{eq:defH+}. 
  We now study the pointwise behaviour of the integrand as $L \to \infty$. Ultimately, via an application of dominated convergence, this will allow us to describe the asymptotic behaviour of the integral.

\subsubsection*{Asymptotics of $\mathbf{1}_{\mathcal{D}_L}$}
    Let us fix any $\mathbf{x} \in H$. We have $\mathbf{x} \in \mathcal{D}_L = H^{+}-L\mathbf{v}$ if and only if 
    \[
        x_i + L v_i
        \ge 
        x_j + L v_j
    \]
    for all $i < j$; therefore, we have
    \begin{align}\label{eq:domoconDL}
        \mathcal{D}_L  \subset \mathcal{C}  :=
    \{
    \mathbf{y} \in H: y_i \ge y_j \text{ if $i < j \le \tfrac{n}{2}$ or $\tfrac{n+1}{2} < i <j$ } \}.
    \end{align}
     Conversely, for $\y \in \mathcal{C}$, we have $y_i + Lv_i \ge y_j + L v_j$ for all $i<j$;
      i.e.\ $\y \in \mathcal{D}_{L}$, 
     for all sufficiently large $L$. Since we study the pointwise behaviour of $\mathbf{1}_{\mathcal{D}_{L}}$, the lower bound on $L$ is allowed to depend on $\y$. 
     Indeed, for $i<j$ with $i < j \le \tfrac{n}{2}$ or $\tfrac{n+1}{2} < i <j$, we have $v_i = v_j$ and $y_i \ge y_j$, and so
      $y_i + L v_i \ge y_j + Lv_j$ holds for all $L$. For the remaining combinations of indices $i<j$, we have  $v_i > v_j$ and so we get $y_i + L v_i \ge y_j + Lv_j$  for all sufficiently large $L$.
    For all $\mathbf{x} \in H$, 
     the pointwise convergence of the indicator function
    \[
    \mathbf{1}_{\mathcal{D}_L}(\x) \to \mathbf{1}_{\mathcal{C}}(\x), 
    \]
    follows as $L \to \infty$.

\subsubsection*{Asymptotics of $f_L$} Recall the notation~\eqref{eq:cartan_decomp}. 
   We observe, as $L \to \infty$, that 
      \begin{equation}\label{eq:defofD+}
          e^{-L} a_{\mathbf{x}+L\mathbf{v}} = \mathrm{diag}(
      e^{x_i + (v_i - 1)L }
      )_{i=1,\ldots,n} \to \mathrm{diag}
      (
      \mathbf{1}_{[v_i=1]} e^{x_i  }
      )_{i=1,\ldots,n} =: D_+(\mathbf{x}) 
      \end{equation}
      and
        \begin{equation}\label{eq:defofD-}
      e^{-L} a_{\mathbf{x}+L\mathbf{v}}^{-1} = \mathrm{diag}(
      e^{-x_i - (v_i + 1)L }
      )_{i=1,\ldots,n} \to \mathrm{diag}
      (
      \mathbf{1}_{[v_i=-1]} e^{-x_i  }
      )_{i=1,\ldots,n} =: D_{-}(\mathbf{x}).
        \end{equation}
      Hence, for all $(k_1,\mathbf{x},k_2)$ not lying on $| X k_1 D_+(\mathbf{x}) k_2|=1$ or $|  k_2^{-1} D_-(\mathbf{x}) k_1^{-1} Y |=1$, we have  
     \begin{equation}
     \label{eq:f(k_1,x,k_2)}
       f_L(k_1,\mathbf{x}+L\mathbf{v},k_2) 
       \to \mathbf{1}_{[| X k_1 D_+(\mathbf{x}) k_2|\le 1]} 
        \mathbf{1}_{[|  k_2^{-1} D_-(\mathbf{x}) k_1^{-1} Y |\le 1]} =: f(k_1,\mathbf{x},k_2) ,
    \end{equation}
      as $L \to \infty$.
      
    For $(k_1,\mathbf{x},k_2)$ lying on $| X k_1 D_+(\mathbf{x}) k_2|=1$ or $|  k_2^{-1} D_-(\mathbf{x}) k_1^{-1} Y |=1$, pointwise convergence 
    of $f_L$ to $f$ 
    may fail: for instance, we may have
    that $\lim_{L \to \infty} | e^{-L} X k_1 a_{\mathbf{x}+L\mathbf{v}} k_2 | =1$ but $| e^{-L} X k_1 a_{\mathbf{x}+L\mathbf{v}} k_2 | > 1 $ for all $L$.
However, the set         \[
       \{(k_1,\x,k_2) :
        | X k_1 D_+(\mathbf{x}) k_2| = 1
        \text{ or }
        |  k_2^{-1} D_-(\mathbf{x}) k_1^{-1} Y | =1
        \}
        \]
        has measure zero.
       To formally justify this, one may cover $K \times K$ by finitely many real 
       analytic coordinate charts. On
each such chart, each entry of the matrices
$Xk_1D_+(\mathbf{x})k_2$ and $k_2^{-1}D_-(\mathbf{x})k_1^{-1}Y$ is real
analytic in $(k_1,\mathbf x,k_2)$. None of these entries is identically equal
to $\pm 1$, as the entries of $D_+(\mathbf x)$ and
$D_-(\mathbf x)$, and thus the entries of $Xk_1D_+(\mathbf{x})k_2$ and $k_2^{-1}D_-(\mathbf{x})k_1^{-1}Y$, can be made arbitrarily small. The vanishing loci of real-analytic functions that are not identically zero, such as $(X k_1 D_+(\mathbf{x}) k_2)_{i,j}\pm1$ and $(k_2^{-1} D_-(\mathbf{x}) k_1^{-1} Y)_{i,j}\pm1$, always constitute a set of measure zero, and our exceptional set is contained in the finite union over $1 \le i,j \le n$ of these sets.
Hence we have pointwise almost everywhere convergence 
$f_L \to f$.

      Moreover, by \eqref{eq:basic_infty_norm_inequ} and the fact that  $| k | \le 1$ for all $k \in K$,  for all $\x \in \mathcal{D}_{L}=H^{+}-L\mathbf{v}$ with $| e^{-L} X k_1 a_{\mathbf{t}} k_2 | \le 1$  and $|e^{-L} (k_1 a_{\mathbf{t}} k_2)^{-1} Y | \le 1$, writing $\mathbf{t}=\mathbf{x}+L\mathbf{v}$, we have 
      \begin{align*}
           e^{x_1} = 
      e^{-L}
      | a_{\mathbf{x} + L\mathbf{v}} | 
      &=
      e^{-L}| (X k_1)^{-1}X k_1 a_{\mathbf{x} + L\mathbf{v}} k_2 k_2^{-1} | \\
      &\le
      n^3 |  X^{-1} |  | e^{-L} X k_1 a_{\mathbf{x} + L\mathbf{v}} k_2  | \le n^3 |  X^{-1} |
      \end{align*}
      and 
      \[
      e^{-x_n} = 
      e^{-L} |a_{\mathbf{-x} - L\mathbf{v}} | 
      \le n^3
      | Y^{-1} |
      | e^{-L}(k_1 a_{\mathbf{t}} k_2)^{-1} Y  | \le  n^3
      | Y^{-1} |.
      \]
       By continuity, analogous bounds hold for $|D_+(\x)|$ and 
      $|D_-(\x)|$. Namely, if we have  $\limsup_{L \to \infty}| e^{-L} X k_1 a_{\mathbf{t}} k_2 | \le 1$  and $\limsup_{L \to \infty}|e^{-L} (k_1 a_{\mathbf{t}} k_2)^{-1} Y | \le 1$, we get 
      \[
      e^{x_1}=
      |D_+(\x)| = \lim_{L\to \infty} ( e^{-L} |a_{\mathbf{x}+  L\mathbf{v}}|) \le n^3|X^{-1}|
      \]
      and 
      \[
      e^{-x_n}=
       |D_-(\x)| = \lim_{L\to \infty} ( e^{-L} |a^{-1}_{\mathbf{x}+  L\mathbf{v}}|) \le n^3|Y^{-1}|.
      \]
    Taking logarithms, it follows that there exists a constant $c_0 = c_0(X,Y) > 0$, independent of $L$, such that 
    \begin{align} \label{eq:suppboundforf}
        x_1 \le c_0, \ x_n \ge - c_0 
    \end{align}
    whenever $\mathbf{1}_{\mathcal{D}_{L}}(\x)f_L(k_1,\mathbf x+L\mathbf{v},k_2)\neq0$ or
$\mathbf{1}_{\mathcal{C}}(\x)f(k_1,\mathbf x,k_2)\neq0$, where $f(k_1,\mathbf x,k_2)$ is defined by~\eqref{eq:f(k_1,x,k_2)}.

\subsubsection*{Asymptotics of $\prod_{i<j} \sinh(x_i-x_j + L(v_i-v_j))$} 
For $v_i > v_j $, since $\sinh(x) = \frac{1}{2}(e^x - e^{-x})$, we see that 
      \begin{align}\label{eq:sinhasymp}
          \sinh(x_i - x_j + L(v_i-v_j)) 
       \sim \frac{e^{x_i - x_j + L(v_i-v_j)}}{2}, \quad \text{as $L \to \infty$,}
      \end{align}
            while for $v_i = v_j$, we have
      \[
      \sinh(x_i - x_j + L(v_i-v_j)) = \sinh(x_i-x_j).
      \]
      Since $\sum_{i<j}(v_i-v_j) = \lfloor n^2/2 \rfloor$, this yields, as $L \to \infty$,   
      \[
      \prod_{i<j} \sinh(x_i - x_j + L(v_i-v_j)) \sim 
      e^{\lfloor n^2/2 \rfloor L} \Psi(\mathbf{x})
      ,
      \]
      where 
      \[
      \Psi(\mathbf{x}) = 2^{-\binom{n}{2}+2\binom{\lfloor n/2 \rfloor}{2}} \prod_{\substack{i < j \\ v_i = v_j}}
      \sinh(x_i-x_j)
      \prod_{\substack{i < j \\ v_i > v_j}} e^{x_i - x_j}.
      \]
    Here we used the fact that there are $\binom{n}{2}-2\binom{\lfloor n/2 \rfloor}{2}$
    factors of the form \eqref{eq:sinhasymp}. Indeed, there
    are $\binom{n}{2}$ indices $i<j$ overall, and of these, there are $2\binom{\lfloor n/2 \rfloor}{2}$ indices  with $v_i = v_j$. As $v_i \ge v_j$ for all $i<j$, it follows that there are $\binom{n}{2}-2\binom{\lfloor n/2 \rfloor}{2}$ indices $i<j$ with $v_i>v_j$.

      We observe that for all $\x \in \mathcal{D}_L$, we have
      \begin{align}\label{eq:domconvpsi}
      0 \le
          e^{-\lfloor n^2/2 \rfloor L}\prod_{i<j} \sinh(x_i - x_j + L(v_i-v_j)) \le \prod_{i<j} e^{x_i-x_j} =\exp(F(\x)),
      \end{align}
      where we set $F(\x)=\sum_{i<j}(x_i-x_j)$.

\subsubsection*{Dominated convergence} 
      Let 
      \[
      \mathcal{C}_0 = 
      \{ 
      \y \in \mathcal{C}:  y_1 \le c_0 , y_n  \ge -c_0
      \},
      \]
      with $\mathcal{C}$ defined as in \eqref{eq:domoconDL} and $c_0$ given as in \eqref{eq:suppboundforf}. By \eqref{eq:domoconDL},
      \eqref{eq:suppboundforf} and 
      \eqref{eq:domconvpsi},  the function 
      $\mathbf{1}_{\mathcal{C}_0} (\mathbf{x}) \exp(F(\x))$ gives an upper bound for the modulus of the renormalised integrand from  \eqref{eq:volEafterchangeofvariables}; i.e. for
      \[
       e^{-\lfloor n^2/2\rfloor L}
       \mathbf{1}_{\mathcal{D}_L}(\mathbf{x})
     f_L(k_1,\mathbf{x}+L\mathbf{v},k_2) 
      \prod_{i<j} \sinh(x_i - x_j + L(v_i-v_j)),
      \]
      as well as for the modulus of its almost everywhere pointwise limit 
      $
      \mathbf{1}_{\mathcal{C}}(\x)
     f(k_1,\mathbf{x},k_2)
      \Psi(\x)$.
      
    We now show that there exists a positive constant $\eta > 0$ with
    \begin{align}\label{eq:fboundonR}
         F(\mathbf{r}) \le - \eta | \mathbf{r} |
    \end{align}
    for all $\mathbf{r} \in \mathcal{R}$, where we let 
    \[ \mathcal{R}:= \{\y \in \mathcal{C}: y_1 \le 0, \; y_n \ge 0 \} .
    \]
     From the definition of $\mathcal{C}$ in \eqref{eq:domoconDL}, for $\y \in \mathcal{R}$ we have that
     $0 \ge y_1 \ge y_i$ for all $i$ with $v_i =1$, and  $0 \le y_n \le y_i$ for all $i$ with $v_i =-1$.
    For $\y \in \mathcal{C}_0$, we have $y_1 \le c_0$ and $y_n \ge -c_0$; by the definition of $\mathcal{C}$ in \eqref{eq:domoconDL}, it follows that 
    $c_0 \ge y_1 \ge y_i$ for all $i$ with $v_i =1$, and  $-c_0 \le y_n \le y_i$ for all $i$ with $v_i =-1$. Therefore every $\x \in \mathcal{C}_0$ may be written as 
    \[
    \x = \x_0 + \mathbf{r}, 
    \]
where $|\x_0| \le c_0$ and $\mathbf{r} \in \mathcal{R}$.
    It therefore suffices
    to establish \eqref{eq:fboundonR}  in order to prove that the function $\mathbf{1}_{\mathcal{C}_0} \exp(F(\x))$ is integrable.
    
    Observe that on the polytope 
    \[
    \mathcal{P} = \{ 
    \mathbf{x} \in H^+: x_1 \le 1, x_n \ge -1 
    \},
    \]
    where $H^{+}$ is defined in \eqref{eq:defH+},
    the vector $\mathbf{v}$  in \eqref{eq:definitionof v} is the unique maximiser of the function $F(\x)$. 
    Indeed, writing $d_i = x_{i} - x_{i+1} \ge 0$, we see that 
    \[
    F(\x) = \sum_{i=1}^{n-1} i(n-i)d_i, \qquad \sum_{i=1}^{n-1} d_i = x_1 - x_n \le 2.
    \]
    Thus, considering the coefficients $i(n-i)$, we see that $F(\x)$ is maximised by  choosing $\x$ in such a way that $d_{n/2} = 2$ for $n$ even and $d_{(n-1)/2}+d_{(n+1)/2}=2$ for $n$ odd. Because of the constraint $ x_1 + \cdots + x_n = 0$, this is achieved on $\mathcal{P}$ exactly by $\x=\mathbf{v}$. 

    For every $\mathbf{r} \in \mathcal{R} \setminus \{\mathbf{0}\}$, there is some small positive number $s_\mathbf{r}$ with $\mathbf{v} + s_\mathbf{r}\mathbf{r} \in \mathcal{P}$. By linearity of $F$ and the fact that $\mathbf{v}$ is the unique maximiser, we have that 
    \[
    F(\mathbf{v} + s_\mathbf{r}  \mathbf{r}) = F(\mathbf{v}) + s_\mathbf{r} F(\mathbf{r}) < F(\mathbf{v}), 
    \]
    and so $F(\mathbf{r}) < 0$. As the set $\{ \mathbf{r} \in \mathcal{R}: | \mathbf{r}| =1 \}$ is compact and as $F$ is linear, the desired inequality  \eqref{eq:fboundonR} follows.

      Now it follows by dominated convergence  that $\lim_{T \to \infty}  (T^{-\lfloor n^2 / 2 \rfloor} \vol(E_{X,Y}(T)))$ exists, is finite and may be expressed as
      \begin{align*}
          \lim_{T \to \infty}  (T^{-\lfloor n^2 / 2 \rfloor} \vol(E_{X,Y}(T)))  
          =
          c_n
          \int_{K} \int_{\mathcal{C}} \int_{K}
          G(k_1,\x,k_2)
          \mathrm{d} k_1
        \mathrm{d} \x
        \mathrm{d} k_2 
        = \kappa(X,Y),
      \end{align*}
      where 
      \[
      G(k_1,\x,k_2) = f(k_1,\mathbf{x},k_2)\Psi(\x).
      \]

\subsubsection*{Positivity of $\kappa(X,Y)$} 
On $\mathcal{C}$, we have $\Psi(\x) \ge 0$ and thus $G(k_1,\x,k_2) \ge 0$, so it suffices to prove that 
      \[
       \int_{K} \int_{W} \int_{K}
          G(k_1,\x,k_2)
          \mathrm{d} k_1
        \mathrm{d} \x
        \mathrm{d} k_2 > 0
      \]
      for some $W \subset \mathcal{C}$. We show that we may take 
      \[
      W = \{
      \y \in H: | \mathbf{z} - \y | \le \mu
      \}, 
      \]
      where $\mu>0$ is some small positive parameter and
$       z_i = - v_i R + \delta_i.
 $
       The $\delta_i$ are chosen to be small perturbations such that $\sum_{i=1}^n \delta_i =0$ and 
      $z_i > z_{i+1}$ if $v_i = v_{i+1}$, and $R$ is a sufficiently large parameter. 
      We recall the definitions of $D_+(\x)$ and $D_-(\x)$ in \eqref{eq:defofD+} and \eqref{eq:defofD-}. 
      By \eqref{eq:basic_infty_norm_inequ}, for $\x \in W$ we have 
      \[
      | X k_1 D_+(\mathbf{x}) k_2| \le
      n^3 | X|  | D_+(\mathbf{x})| \leq
      n^3 | X| e^{-R + \delta_1 + \mu }
      \]
      and
      \[
      |  k_2^{-1} D_-(\mathbf{x}) k_1^{-1} Y|\le n^3 | Y| |   D_-(\mathbf{x})| \leq  n^3 | Y| e^{-R -\delta_n + \mu}.
      \]
      Thus we may choose $R$ sufficiently large and $\mu$ sufficiently small in such a way that $W$ lies inside the interior of $\mathcal{C}$ and $K \times W \times K$ is contained in the support of 
      \[
      \mathbf{1}_{[| X k_1 D_+(\mathbf{x}) k_2|\le 1]}
      \mathbf{1}_{[|  k_2^{-1} D_-(\mathbf{x}) k_1^{-1} Y|\le 1]}
      .
      \]
      Now let $ \nu >0$ be the minimum of the function $\Psi$ on the compact set $W$. Then clearly 
      \[
      c_n 
       \int_{K} \int_{W} \int_{K}
          G(k_1,\x,k_2)
          \mathrm{d} k_1
        \mathrm{d} \x
        \mathrm{d} k_2 \ge c_n  \, \nu \, \mu_H(  W )>0. \qedhere
      \]
\end{proof}

\subsection{Neighbourhoods of the identity}

 For $\varepsilon > 0$, we define 
\begin{align}\label{eq:nbh_id}
    \mathcal{O}_\varepsilon = 
    \{ 
    h \in \SL_n(\R):
    | 
    I_n - h
    | < \varepsilon, \ 
    | 
    I_n - h^{-1}
    | < \varepsilon
    \}.
\end{align}
Recall the definition \eqref{eq:OO} of 
$\mathcal{O}_{\varepsilon}\mathcal{O}_{\varepsilon}$.
The family $(\mathcal{O}_\varepsilon)_{\varepsilon>0}$ consists of symmetric neighbourhoods of $I_n$, with the properties that $\mathcal{O}_{\varepsilon_1} \subset \mathcal{O}_{\varepsilon_2}$ if $\varepsilon_1 < \varepsilon_2$ and that the projection $\mathcal{O}_{\varepsilon}\mathcal{O}_{\varepsilon} \to \SL_n(\R) / \SL_n(\Z)$ is injective for all sufficiently small $\varepsilon$. 
Indeed, the symmetry of the sets as well as the monotonicity in  $\varepsilon$ follow directly from the definition of $\mathcal{O}_\varepsilon$, while we give a short proof of the injectivity in the next result.

\begin{lem}
     The projection 
     $$
     \mathcal{O}_{\varepsilon}\mathcal{O}_{\varepsilon} \to \SL_n(\R) / \SL_n(\Z)
     $$ 
     is injective for all $\varepsilon \in \big(0,\frac{1}{(1+2n)^2}\big)$. 
\end{lem}
\begin{proof}
    Assume  there are $g,h \in \mathcal{O}_\varepsilon \mathcal{O}_{\varepsilon}$ such that $[h]_{\SL_n(\R) / \SL_n(\Z)} = [g]_{\SL_n(\R) / \SL_n(\Z)}$; i.e.\ such that $ h^{-1}g \in \SL_n(\Z)$. 
    We need to prove that $h =g$.
    
    Write $g = g_1 g_2$, $h = h_1 h_2$ for $g_1,g_2,h_1,h_2 \in \mathcal{O}_\varepsilon$. 
    Using \eqref{eq:basic_infty_norm_inequ}, we have  
\[
    | I_n - g_1 g_2 | \le | I_n - g_1| + | g_1 - g_1 g_2| \le \varepsilon + n | g_1 | | I_n - g_2| \le (1 + 2n) \varepsilon,
\]
and similarly $|
I_n - (h_1 h_2)^{-1}| \le (1+2n)\varepsilon$. But now, another application of the inequalities from \eqref{eq:basic_infty_norm_inequ} yields that 
\[
    | I_n - h^{-1}g | \le
    | I_n - h^{-1} | + n|h^{-1}| | I_n - g | \le 
    (1+2n)^2\varepsilon
    < 1;
\]
as  $h^{-1}g \in \SL_n(\Z)$, the inequality $ | I_n - h^{-1}g | < 1 $ implies that $| I_n - h^{-1}g | =0$, whence $g =h$.
\end{proof}

\subsection{Well-roundedness of $E_{X,Y}(T)$} 
In this section, we verify that the family of sets $(E_{X,Y}(T))_{T>0}$ satisfies the required well-roundedness property
from Definition  \ref{def:wellrounded}.

\begin{lem}\label{lem:wellroundedness}
    The sets $(E_{X,Y}(T))_{T>0}$ are well-rounded with respect to the family of neighbourhoods $(\mathcal{O}_\varepsilon)_{\varepsilon>0}$ defined in \eqref{eq:nbh_id}.
\end{lem}
 \begin{proof}
     We may assume without loss of generality that $\varepsilon < 1$. 
Assume that $h_1, h_2 \in \mathcal{O}_\varepsilon$, and let $g \in E_{X,Y}(T)$. By \eqref{eq:basic_infty_norm_inequ}, we have
\begin{align*}
    | X h_1 g h_2|
    &\le 
    | X (I_n-h_1) g h_2| + 
     | X g (I_n-h_2)|
     + | Xg | \\
     &\le 
     2 n^4 |X^{-1}|  |X| \varepsilon T 
     + 
     n \varepsilon T
     + T  \\
     & = (1+\varepsilon q)T,
\end{align*}
for some positive constant $q = q(n,X,Y)>0$.
Here we used the fact that we have  $|h|, \; |h^{-1}| \le 1 + \varepsilon < 2$ for all $h \in \mathcal{O}_\varepsilon$, and that $| g | \le n |X^{-1}| |Xg| \le n |X^{-1}|  T$.
Analogously, after possibly increasing $q$, one may show that 
\[
    |  (h_1 g h_2)^{-1} Y| \le (1+\varepsilon q)T.
\]
Recalling the notation in Definition~\ref{def:wellrounded}, we have that 
\begin{align}\label{eq:B+subset}
\mathcal{O}_\varepsilon E_{X,Y}(T) \mathcal{O}_\varepsilon  =  (E_{X,Y}(T))^{+}(\varepsilon) \subset E_{X,Y}((1+\varepsilon q)T).
\end{align}
Shrinking $\varepsilon$ further if necessary, we may assume $\varepsilon < 1/q$.

Conversely, let $g' \in E_{X,Y}((1-\varepsilon q)T)$. Then the same argument  gives that 
\[
    | X h_1^{-1} g' h_2^{-1}|, \;
     |  (h_1^{-1} g' h_2^{-1})^{-1} Y| \le (1+\varepsilon q)(1-\varepsilon q)T < T;
\]
this means that $h_1^{-1} g' h_2^{-1} \in E_{X,Y}(T)$, i.e. that $ g' \in h_1 E_{X,Y}(T) h_2$ for all $h_1, h_2 \in \mathcal{O}_\varepsilon$ and thus
\begin{align}\label{eq:B-inclusion}
E_{X,Y}((1-\varepsilon q)T)
\subset
    \bigcap_{h_1, h_2 \in \mathcal{O}_\varepsilon} (h_1 E_{X,Y}(T) h_2) = (E_{X,Y}(T))^-(\varepsilon)   .
\end{align}

Applying \eqref{eq:B+subset}
and \eqref{eq:B-inclusion} in combination with
Lemma~\ref{lem:volumecomputation}, we get that 
\begin{align*}
    \frac{\vol((E_{X,Y}(T))^{-}(\varepsilon))}{\vol((E_{X,Y}(T))^{+}(\varepsilon))} \ge \frac{\vol(E_{X,Y}((1-\varepsilon q)T))}{\vol(E_{X,Y}((1+\varepsilon q)T))} \sim \left(
    \frac{1-\varepsilon q}{1+\varepsilon q}
    \right)^{\lfloor n^2/2 \rfloor}
\end{align*}
as $T \to \infty$. As $\varepsilon \to 0$, the right-hand side approaches 1. This proves that $E_{X,Y}(T)$ is well-rounded with respect to $(\mathcal{O}_\varepsilon)_{\varepsilon>0}$.
 \end{proof}
\subsection{Ergodic theorem for $E_{X,Y}(T)$}\label{section:ergodic}
By \cite[Theorem~4.5 and Remark~4.6(ii)]{gorodnik_nevo}, with the notation of Definition~\ref{def:meanergodictheorem}, for almost simple non-compact connected Lie groups $G$ and any lattice $\Gamma \subset G$, we have for any $\delta>0$ that  
  \begin{align*}
           \left\|
           \pi_X(\beta_T)(f) -  
           \int_X
           f \mathrm{d} \mu
     \right\|_{L^2(X)} \le C_\delta m_G(B_T)^{-\frac{1}{2n_e(p)}+\delta} \|f\|_{L^2},
       \end{align*}
       with \[
       n_e(p)=
\begin{cases}
\text{the least even integer greater than or equal to }p/2 &\text{ if } p>2,\\
1 & \text{ if } p=2,
\end{cases}\]        
       provided that the representation on $L^2_0(G/\Gamma)$ is an $L^{p+}$-representation. By \cite[Example~4.3]{gorodnik_nevo}, this holds in our case with $p = 2(n-1)$. As $m_G(E_{X,Y}(T))\to \infty$ as $T \to \infty$, by Lemma~\ref{lem:volumecomputation}, we see that the mean ergodic theorem holds for $(E_{X,Y}(T))_{T>0}$. 
       
       Recalling from \eqref{eq:B+subset} and \eqref{eq:B-inclusion} the notation for 
       $(E_{X,Y}(T))^{+}(\varepsilon)$ and $(E_{X,Y}(T))^{-}(\varepsilon)$, it remains to argue that the ergodic theorem is stable. 
       Since, for any $\varepsilon>0$, we have the inclusion $E_{X,Y}(T) \subseteq (E_{X,Y}(T))^{+}(\varepsilon)$, it follows that \[
       m_G((E_{X,Y}(T))^{+}(\varepsilon)) \to \infty
       \] as $T \to \infty$. By Lemma~\ref{lem:wellroundedness}, this further implies that $m_G((E_{X,Y}(T))^{-}(\varepsilon)) \to \infty$ as $T \to \infty$. Therefore the ergodic theorem is stable.

\subsection{Applying Theorem~\ref{thm:gorodnik_nevo}}
In this section, we use  Theorem~\ref{thm:gorodnik_nevo} to prove Theorem~\ref{lem:maincountinglemma}.
\begin{proof}[Proof of Theorem~\ref{lem:maincountinglemma}]
    By Lemma~\ref{lem:wellroundedness}, the sets $(E_{X,Y}(T))_{T>0}$ are well-rounded with respect to the family of neighbourhoods $(\mathcal{O}_\varepsilon)_{\varepsilon>0}$ defined in \eqref{eq:nbh_id}. By our discussion in Section~\ref{section:ergodic}, the stable mean ergodic theorem holds for $(E_{X,Y}(T))_{T>0}$. Thus we may apply Theorem~\ref{thm:gorodnik_nevo} to conclude that 
    \[
     \# (E_{X,Y}(T) \cap \SL_n(\Z)) \sim \frac{\vol(E_{X, Y}(T))}{\vol (\SL_n(\R) / \SL_n(\Z))},
    \]
    as $T \to \infty$.
    By Lemma~\ref{lem:volumecomputation} and \eqref{eq:siegelnormalisation}, we finally get 
    \begin{equation*}
        \# (E_{X,Y}(T) \cap \SL_n(\Z)) \sim  \kappa (X,Y) T^{\lfloor n^2/2 \rfloor}.
     \qedhere
    \end{equation*}    
\end{proof}

\section{Asymptotics } \label{sec:asymp}

\subsection{Deduction of the main asymptotic formula}

In this section we use the results from  Section~\ref{section:latticepts} in order to prove Theorem~\ref{thm:asymptotic_formula}.

\begin{proof}[Proof of Theorem~\ref{thm:asymptotic_formula}]
Let $M\in \mathcal{M}_n(\ZZ)$ be non-singular and suppose that  $M = AB$ for some matrices $A,B \in \mathcal{M}_{n}(\Z)$. 
Define the lattice $\mathsf{\Lambda} =A\Z^n$. Then clearly  
    \[
    M \Z^n = AB\Z^n \subseteq \mathsf{\Lambda}  \subseteq \Z^n. 
    \]
    There are only finitely many possibilities for the lattice $\mathsf{\Lambda}$. (See 
    Proposition \ref{prop:latticeinclusion}.)
    
    For each such lattice $\mathsf{\Lambda}$
    with $M\Z^n \subseteq \mathsf{\Lambda} \subseteq \Z^n$, let us fix a basis matrix $B_\mathsf{\Lambda}$. 
       Note that the matrix $B_{\mathsf\Lambda}^{-1}M$ has integer entries since 
$M \Z^n \subseteq \mathsf{\Lambda}$.
 For $A$ with $M= AB$, there is a lattice $\mathsf{\Lambda}$
    and a change of basis matrix $U \in \GL_n(\Z)$
    such that 
   $
    A =  B_\mathsf{\Lambda} U
    $ and hence $B = (B_\mathsf{\Lambda} U)^{-1}M$. Therefore we get that 
    \[
    \tau_n(T,M) 
    =
    \sum_{M\Z^n \subseteq \mathsf{\Lambda} \subseteq \Z^n} \# \{ 
    U \in \GL_n(\Z): | B_\mathsf{\Lambda} U | \le T, \: | U^{-1}B_\mathsf{\Lambda}^{-1}M|\le T
    \}.
    \]
    Let $J = \mathrm{diag}(-1,1,\ldots,1)$. Then we have $\det(UJ) = - \det(U)$ but
    \[
    | B_\mathsf{\Lambda} U | = |  B_\mathsf{\Lambda} U J | \quad \text{ and } 
    \quad 
    | (B_\mathsf{\Lambda} U)^{-1}M| =
    | (B_\mathsf{\Lambda} U J)^{-1}M|.
    \]
    Hence we may restrict our count to $\SL_n(\Z)$ and conclude that 
    \begin{equation}\label{eq:chicken}
        \tau_n(T,M) 
    =
    2
    \sum_{M\Z^n \subseteq \mathsf{\Lambda} \subseteq \Z^n} \# \{ 
    U \in \SL_n(\Z):  | B_\mathsf{\Lambda} U | \le T, \: | U^{-1}B_\mathsf{\Lambda}^{-1}M|\le T
    \}. 
    \end{equation}
But then it follows that 
\begin{align*}    
     \tau_n(T,M) 
    &=
    2
    \sum_{M\Z^n \subseteq \mathsf{\Lambda} \subseteq \Z^n} \# (E_{B_\mathsf{\Lambda}, B_\mathsf{\Lambda}^{-1}M}(T)\cap \SL_n(\Z)) 
    \\
    & \sim 2 T^{\lfloor n^2/2 \rfloor}
    \sum_{
    \substack{
    M \Z^n \subseteq \mathsf{\Lambda} \subseteq \Z^n, \\
    \mathsf{\Lambda} \text{ lattice }
    }
    }
    \kappa (B_\mathsf{\Lambda}, B_\mathsf{\Lambda}^{-1}M),
    \end{align*}
by  Theorem~\ref{lem:maincountinglemma}. 
\end{proof}

\subsection{The special case of $2\times 2$ matrices}\label{2x2}

    Let us consider the special case $n =2$ and $M = I_2$. In this section we verify that our formula in 
    Corollary~\ref{cor1.3} gives the same asymptotics  as \eqref{eq:krieg}.
Recalling the notation 
\eqref{eq:divisor-id},
it follows from Corollary~\ref{cor1.3}  that 
    \[
    \tau_2(T) 
 \sim 2 T^{2} \kappa(I_2, I_2),
    \]
    as $T \to \infty$. By the proof of Lemma~\ref{lem:volumecomputation}, we have  
    \[
    \kappa(I_2, I_2) = c_2 \int_{K} \int_{\mathcal{C}} \int_{K}  
      \mathbf{1}_{[| k_1 D_+(\mathbf{x}) k_2|\le 1]}
        \mathbf{1}_{[|  k_2^{-1} D_-(\mathbf{x}) k_1^{-1} |\le 1]}
      \Psi(\x) \mathrm{d} k_1
      \mathrm{d} \mathbf{x}
      \mathrm{d} k_2,
    \]
    where 
    \[
    D_+(\x) = \begin{pmatrix}
        e^{x_1} & 0 \\
        0 & 0
    \end{pmatrix}, \quad  
    D_-(\x) = \begin{pmatrix}
        0 & 0 \\
        0 & e^{-x_2}
    \end{pmatrix}, \quad \Psi(\x) = \frac{e^{x_1 - x_2}}{2},
    \]
with   $K = \SO_2(\R)$  and 
    $
    \mathcal{C} = \{ 
    \x \in \R^2:  x_1 + x_2 = 0
    \}.
    $

    Writing 
    \[
    k_1 = \begin{pmatrix}
        \cos (\alpha_1) & -\sin (\alpha_1) \\
        \sin(\alpha_1) & \cos(\alpha_1)
    \end{pmatrix}, \quad 
    k_2 = \begin{pmatrix}
        \cos (\alpha_2) & -\sin (\alpha_2) \\
        \sin(\alpha_2) & \cos(\alpha_2)
    \end{pmatrix},
    \]
    one may compute for all $\x \in \mathcal{C}$ that  
    \begin{align*}
          |k_1 D_+(\x) k_2| 
          = |  k_2^{-1} D_-(\mathbf{x}) k_1^{-1} | 
          = e^{x_1} m(\alpha_1)m(\alpha_2), 
    \end{align*}
    where 
    $m(\alpha_i) = \max 
          \{
          |\cos(\alpha_i)|, |\sin(\alpha_i)|
          \}$  is a $\tfrac{\pi}{2}$-periodic even function. 
    Indeed, we have 
    \begin{align*}
         k_1 D_+(\x) k_2
    &= 
    \begin{pmatrix}
        \cos (\alpha_1) & -\sin (\alpha_1) \\
        \sin(\alpha_1) & \cos(\alpha_1)
    \end{pmatrix}
     \begin{pmatrix}
        e^{x_1} & 0 \\
        0 & 0
    \end{pmatrix}
    \begin{pmatrix}
        \cos (\alpha_2) & -\sin (\alpha_2) \\
        \sin(\alpha_2) & \cos(\alpha_2)
    \end{pmatrix} \\
    &=
    e^{x_1}
    \begin{pmatrix}
        \cos (\alpha_1)
        \cos (\alpha_2)& -\cos (\alpha_1)\sin (\alpha_2) \\
        \sin (\alpha_1)\cos (\alpha_2) & -\sin (\alpha_1)\sin (\alpha_2)
    \end{pmatrix}
    \end{align*}
   and 
          \begin{align*}
         k_2^{-1} D_{-}(\x) k_1^{-1}
    &= 
    \begin{pmatrix}
        \cos (\alpha_2) & \sin (\alpha_2) \\
        -\sin(\alpha_2) & \cos(\alpha_2)
    \end{pmatrix}
     \begin{pmatrix}
        0 & 0 \\
        0 & e^{-x_2}
    \end{pmatrix}
     \begin{pmatrix}
        \cos (\alpha_1) & \sin (\alpha_1) \\
        -\sin(\alpha_1) & \cos(\alpha_1)
    \end{pmatrix}
    \\
    &=
    e^{-x_2}
    \begin{pmatrix}
        -\sin (\alpha_2)
        \sin (\alpha_1)& \sin (\alpha_2)\cos (\alpha_1) \\
        -\cos (\alpha_2)\sin (\alpha_1) & \cos (\alpha_2)\cos (\alpha_1)
    \end{pmatrix}.
    \end{align*}    
    By \eqref{eq:cartanvol} and the fact that $m$ is $\tfrac{\pi}{2}$-periodic and even, we get 
    \begin{align*}
        \kappa (I_2,I_2) &=  \frac{c_2}{2 (2 \pi)^2} \int_{0}^{2 \pi}  \int_0^{2 \pi} 
        \int_{-\infty}^{-\log(m(\alpha_1)m(\alpha_2))}
       e^{2x_1}
          \mathrm{d} x_1
        \mathrm{d} \alpha_1
      \mathrm{d} \alpha_2 \\
      &= 
      \frac{c_2}{4 (2 \pi)^2} \int_{0}^{2 \pi}  \int_0^{2 \pi} 
        (m(\alpha_1)
        m(\alpha_2))^{-2} 
        \mathrm{d} \alpha_1
      \mathrm{d} \alpha_2 \\
      & =
      \frac{4 c_2}{ \pi^2} \int_{0}^{ \pi/4}  \int_0^{\pi/4} 
        (\cos(\alpha_1)
        \cos(\alpha_2))^{-2}
        \mathrm{d} \alpha_1
      \mathrm{d} \alpha_2 
      \\
      & = 
      \frac{4 c_2}{ \pi^2},
    \end{align*}
since $\int_0^{\pi/4} \frac{\mathrm{d}x}{\cos(x)^2} =1$. 
It follows from \eqref{eq:cartanvol} that 
    \[
     c_2 =  
      \frac{
      2^{\binom{2}{2}+2-1} \pi^{\frac{2(2+1)}{2}-1}
      }{ 
       \Gamma(\tfrac{2}{2})^{2} \zeta(2)} = \frac{4 \pi^{2}}{{\pi^2}/{6}} = 24.
    \]
    This shows 
    $\kappa (I_2, I_2) = {96}/{\pi^2}$, which thereby shows that 
    Corollary~\ref{cor1.3} agrees with  \eqref{eq:krieg}
when  $n =2$ and $M = I_2$.

\section{Uniform upper bound}

In this section we prove Theorem \ref{thm:uniform}. 
We may assume that $n \ge 2$ throughout, since the case $n=1$ follows from the usual divisor bound.  We recall our convention that any implied constant is allowed to depend on $n$. 
 
\subsection{Notation}
For symmetric matrices
$A,B \in \mathcal{M}_n(\R)$, we write 
\begin{equation}\label{def:matrixorder}
     A \preceq B
\end{equation}
whenever the difference $B-A$ is positive semidefinite. Equivalently, this means that 
\begin{equation}\label{eq:matrixordereuclideannorm}
        \x^t A \x \le \x^t B \x,\quad  \text{ for all }  \x \in \R^n.
\end{equation}
For {\it positive  definite} symmetric matrices $A,B \in \mathcal{M}_n(\R)$, we claim that 
\begin{equation}\label{eq:matrixorderimpliesdet}
    A \preceq B \quad \Rightarrow \quad \det(A) \le \det(B).
\end{equation}
It is clear that the relation is preserved under conjugation, since this corresponds to a change of basis of the underlying quadratic form in \eqref{eq:matrixordereuclideannorm}. In particular, this means that $A \preceq B$ is equivalent to $C\preceq I_n$, where
\[
   C:= B^{-1/2} A B^{-1/2}.
\]
Note that $B^{1/2}$ is well-defined, since $B$ is a symmetric positive definite matrix. 
The matrix $C$ is a positive definite real symmetric matrix, so all of its eigenvalues are positive real numbers. Let $\lambda$ be any eigenvalue of $C$ with eigenvector $\mathbf{v}$. Then, since $C \preceq I_n$, we get 
$ \lambda \| \mathbf{v} \|_2^2 \le  \| \mathbf{v} \|_2^2$, whence $\lambda \in (0,1]$. It follows that 
\[
    \det(A) \det(B)^{-1} =  \det(C) \le 1,
\]
as required to complete the proof of \eqref{eq:matrixorderimpliesdet}.

\subsection{Intermediate lattices}
\label{section:intermediate_lattices}

We now give an upper bound for the number of lattices $\mathsf{\Lambda}\in \mathcal{L}_n$ satisfying the inclusions $M\Z^n\subseteq \mathsf{\Lambda} \subseteq \Z^n$, where $M$ is an integer matrix with $\det(M) \neq 0$ and  $\mathcal{L}_n$ denotes the set of all integer lattices of full rank.

\begin{prop}\label{prop:latticeinclusion}
    Let $M \in \mathcal{M}_n(\Z)$ with $\det(M) \neq 0$. Then 
    \[
    \# \{ \mathsf{\Lambda} \in \mathcal{L}_n: M \Z^n \subseteq \mathsf{\Lambda} \subseteq \Z^n \} \ll_{\varepsilon} |\det(M)|^{ \lfloor n^2 /4 \rfloor/n + \varepsilon } ,
    \]
    for any $\eps>0$.
\end{prop}

 We will use the following group-theoretic result in order to deduce the bound stated in Proposition~\ref{prop:latticeinclusion}.
 
\begin{lem}\label{lem:countingsubgroups}
    Let $G$ be a finite abelian group generated by at most $n$ elements. Then the number of subgroups $H \le G$ is  $O_{\varepsilon,n} ((\# G)^{\lfloor n^2 /4 \rfloor/n + \varepsilon})$.
\end{lem}

\begin{proof}
     Let us first consider a finite abelian $p$-group $G_p$  of order $\# G_p = p^{m_p}$ that is generated by at most $n$ elements. Write 
     $G_p = C_{p^{e_1}} \times \cdots \times C_{p^{e_r}}$ for 
     $1 \le e_1 \le \cdots \le e_r$.
    Since
     we assume that $G_p$ is generated by at most $n$ elements, it follows that $r \le n$. 
     Put 
     \[
     f_1 = e_1, \qquad f_j = e_j - e_{j-1} \quad \text{ for $j \ge 2$},  
     \]
     and  put 
     $
     \alpha_k = \lfloor k^2/4 \rfloor.
     $
     Then, on  taking $a=0$ in \cite[Theorem~3]{bhowmik_wu_zeta2001}, the number of subgroups of $G_p$ is a polynomial in $p$ with non-negative integer coefficients of degree 
     \[
     D(G_p) = \sum_{j=1}^r \alpha_{r-j+1}f_j,
     \]
      and the sum of its coefficients is 
     \[
     \prod_{j=1}^r \bigg( 1+ 
     \sum_{k=1}^j f_k
     \bigg) = \prod_{j=1}^r (e_j +1) \le (1+m_p)^n.
     \]
     We claim that  $k \mapsto \alpha_k/k$ is an increasing function. Indeed, we have 
     $\alpha_{2m}/(2m) =m/2$ and $\alpha_{2m+1}/(2m+1) =m(m+1)/(2m+1)$, and clearly
     $$
     \frac{m}{2}\leq \frac{m(m+1)}{2m+1}\leq\frac{m+1}{2}.
     $$
     Thus it follows that 
     \[
     D(G_p) \le \frac{\alpha_n}{n}
     \sum_{j=1}^r (r-j+1)f_j = \frac{\alpha_n}{n} m_p.
     \]
     It follows that the number of subgroups of $G_p$ is  
     \[
     \# \{ 
     H_p \le G_p
     \} \le \prod_{j=1}^r (1+e_j) p^{D(G_p)} \le (m_p+1)^n (\# G_p)^{\alpha_n / n}. 
     \]

     Now, for any finite abelian group $G$ that is generated by at most $n$ elements, write 
     \[
     G = \prod_{p \mid \#  G} G_p, \qquad \# G_p = p^{m_p}. 
     \]
Since      subgroups decompose uniquely over Sylow factors, we obtain
     \[
     \# \{ H \le G \}
     \le (\# G)^{\alpha_n/n}  \prod_{p | \#  G} (1+m_p)^n \le (\# G)^{\alpha_n/n} \tau (\# G)^n \ll_{\varepsilon,n} (\# G)^{\alpha_n/n + \varepsilon},
     \]   
     on applying the divisor sum estimate.
\end{proof}

\begin{proof}[Proof of Proposition~\ref{prop:latticeinclusion}]
    Since the lattices $M\Z^n\subseteq \mathsf{\Lambda} \subseteq \Z^n$ are in bijection with the subgroups of $\Z^n/M\Z^n$ via the map $\mathsf{\Lambda} \mapsto \mathsf{\Lambda}/M\Z^n$, an application of Lemma~\ref{lem:countingsubgroups} to the group $\Z^n/M\Z^n$, which is generated by at most $n$ elements and has order $\#(\Z^n/M\Z^n) = |\det(M)|$, immediately implies Proposition~\ref{prop:latticeinclusion}.    
\end{proof}

\subsection{Uniform bounds for lattice points on
  \texorpdfstring{$\SL_n(\R)$}{SLₙ(R)}}
\label{section:boundingnumberofbases}

In this section, we prove the following result, which may be of independent interest. 

\begin{thm}\label{thm:uniformlatticecount}
Let $g,h \in \GL_n(\R)$ with $|\det(g)|=|\det(h)|=1$ and let $T_1, T_2 \ge 0$. Then 
    \[
    \# 
    \left\{
    U \in \SL_n(\Z):~|gU| \le T_1,~|(gU)^{-1}h|\le T_2 
    \right\} \ll_{\varepsilon} (T_1 T_2)^{\lfloor n^2/4 \rfloor + \varepsilon} ,
    \]
    for any $\eps>0$.
\end{thm}

Whenever $T_1 < (n!)^{-1/n}$ or $T_2 < (n!)^{-1/n}$, the set 
\[
\left\{
    U \in \SL_n(\Z):~|gU| \le T_1,~|(gU)^{-1}h|\le T_2 
    \right\} 
\]
is empty, 
so it suffices to prove Theorem~\ref{thm:uniformlatticecount} for  $T_1,T_2 \ge (n!)^{-1/n}$. Indeed, taking $A \in \{gU, (gU)^{-1}h \}$, we have $1 = |\det(A)| \le n! |A|^n$ and thus 
\begin{equation}\label{eq:lowerboundsuchthatEhnonempty}
    |A|\ge (n!)^{-1/n}.
\end{equation}

We first bound the volume of a related subset of $\SL_n(\R)$. For $h \in \GL_n(\R)$ and parameters $T_1,T_2 \ge 0$, we set 
\[
    E_h(T_1, T_2) = 
    \{
    g \in \SL_n(\R):~|g|\le T_1, ~ |g^{-1}h|\le T_2
    \}.
\]
Comparing with $E_{X,Y}(T)$, introduced in \eqref{def:kappa}, we note that
$E_{I_n,h}(T) = E_h(T,T)$.
By covering each $U \in \SL_n(\Z)$ such  that $|gU|\le T_1$ and $|(gU)^{-1}h| \le T_2$ 
by disjoint patches of the same positive volume, we will be able to  compare the discrete count in Theorem~\ref{thm:uniformlatticecount} with the volume of $E_{h}(T_1,T_2)$.

\begin{lem}\label{lem:volumeuniformbound}
    Let $h \in \GL_n(\R)$ with $|\det(h)| =1$ and let $T_1,T_2\geq 0$. Then 
    \[
    \vol (E_h(T_1,T_2)) \ll_{\varepsilon} (T_1 T_2)^{\lfloor n^2/4 \rfloor+\varepsilon},
    \]
    for any $\eps>0$.
\end{lem}

\begin{proof}
    After possibly replacing $h$ by $h \mathrm{diag}(-1,1,\ldots,1)$, which leaves the value of $| g^{-1} h |$ invariant, we may assume that $h \in \SL_n(\R)$. 
    Moreover, as in  \eqref{eq:lowerboundsuchthatEhnonempty}, we see that $E_h(T_1,T_2)$ is empty   whenever $ T_1 < (n!)^{-1/n}$ or $T_2 < (n!)^{-1/n}$. Hence we may assume that $T_1, T_2 \ge (n!)^{-1/n}$. This implies that $\log(nT_1)>0$  and $\log(nT_2)>0$, 
since      $(n!)^{-1/n} > n^{-n/n} =n^{-1}$ for $n \ge 2$, 
    a
fact that we shall tacitly use throughout the proof. 
   
    Let us write $g = k_1 a_{\mathbf{t}}k_2$, as in \eqref{eq:cartan_decomp}, where $k_1, k_2 \in K$ and $\mathbf{t} =(t_1,\ldots, t_n) \in H^+$. The condition $ |g| \le T_1$ implies  that
    \begin{equation}\label{eq:condfromT1}
        t_1 \le \log(n^2T_1).
    \end{equation}
    Indeed, for $g \in E_h(T_1,T_2)$,  it follows from \eqref{eq:basic_infty_norm_inequ} that 
    \[
    \frac{1}{n^2} e^{t_1} = \frac{1}{n^2} |a_\mathbf{t}| \le |k_1 a_\mathbf{t}k_2| = |g| \le T_1.
    \]

For $g \in E_h(T_1,T_2)$ and for all $\y \in \R^n$, we have 
    \[
    \| (g^{-1}h)^t \y \|_2^2 = \y^t (g^{-1}h)(g^{-1}h)^t\y  \le (nT_2)^2  \| \y \|_2^2;
    \]
    choosing $\y = g^{t}\x$ gives  the condition
    \[
    (nT_2)^{-1} \| h^t \x \|_2 \le \| g^t\x\|_2 ,
    \]
    for all $\x \in \R^n$.   
    Writing $g = k_1 a_\mathbf{t} k_2$ as before, and decomposing $h$ via \eqref{eq:cartan_decomp} as $h = \ell_1 a_{\mathbf{s}} \ell_2$, gives
    \[
    \| g^t \x \|_2 = \| k_2^t a_{\mathbf{t}} k_1^t \x \|_2 = \| a_{\mathbf{t}} k_1^t \x \|_2,
    \qquad  \| h^t \x \|_2 = \| a_{\mathbf{s}} \ell_1^t \x \|_2.
    \]
    Here we used the fact that $k_2 \in K =\SO_n(\R)$, so that $k_2 k_2^t = I_n$.
It follows that the condition $
     (nT_2)^{-1}\|h^t\x  \|_2 \le \| g^t \x\|_2$ is equivalent to $(nT_2)^{-1}\|a_{\mathbf{s}} \ell_1^t \x \|_2 \le \| a_{\mathbf{t}} k_1^t \x\|_2.$
     Under the change of variables $\z = \ell_1^t \x$, this yields 
     \[
     (nT_2)^{-1}\|a_{\mathbf{s}} \z \|_2 \le \| a_{\mathbf{t}} k_1^t (\ell_1^{-1})^{t} \z \|_2.
     \]
     In the notation introduced in \eqref{def:matrixorder}, this can  equivalently be expressed as 
     \begin{equation}\label{eq:condfromT2}
              (nT_2)^{-2} a_{\mathbf{s}}^2 \preceq \ell_1^{-1} k_1 a_{\mathbf{t}}^2 k_1^t (\ell_1^{-1})^{t}.
     \end{equation}

    Using the fact that $K$ is equipped with a probability measure, so that $\vol (K)=1$, it follows from combining \eqref{eq:cartanvol}  with \eqref{eq:condfromT1} and \eqref{eq:condfromT2} that
\begin{align}
    \vol(E_h(T_1,T_2)) 
    &\ll \int_K \int_{H^{+}}\int_K \mathbf{1}_{[ k_1 a_{\mathbf{t}} k_2 \in E_h(T_1,T_2)]} \prod_{i<j} \sinh(t_i - t_j) \mathrm{d} k_2 \mathrm{d}\mathbf{t}\mathrm{d}k_1  \nonumber
    \\ 
    &\ll
     \int_K \int_{H^{+}}  \mathbf{1}_{[  (nT_2)^{-2} a_{\mathbf{s}}^2 \preceq \ell_1^{-1} k_1 a_{\mathbf{t}}^2 k_1^t (\ell_1^{-1})^{t} ]} \mathbf{1}_{[t_1 \le  \log(n^2 T_1)]}\prod_{i<j} \sinh(t_i - t_j)\mathrm{d}\mathbf{t}\mathrm{d}k_1 \nonumber
     \\
     &=
      \int_{H^{+}} 
      \mathbf{1}_{[t_1 \le \log(n^2T_1)]} \prod_{i<j} \sinh(t_i - t_j)
      \int_K  \mathbf{1}_{[ (nT_2)^{-2} a_{\mathbf{s}}^2 \preceq  k a_{\mathbf{t}}^2 k^t ]} \mathrm{d}k \mathrm{d}\mathbf{t}, \label{eq:volumeEh_intermediatestep}
\end{align}
where, in the final line, we exchanged the order of integration and replaced $\ell_1^{-1} k_1$ by $k$ via a measure-invariant change of variables. Observe that the factors occurring in the density are of the form
\[
    \sinh(t_i - t_j) = \frac{e^{t_i-t_j}-e^{t_j-t_i}}{2} \ge 0.
\]

We proceed by  showing that 
\begin{equation}\label{eq:Kintegraljbound}
        \int_K  \mathbf{1}_{[ (nT_2)^{-2} a_{\mathbf{s}}^2 \preceq  k a_{\mathbf{t}}^2 k^t ]} \mathrm{d}k  \le \min \left(1, \min_{1\le j \le n-1} ((nT_2)^{n(n-j)} e^{-n(t_1 + \cdots + t_j)}) \right),
\end{equation}
a bound that is  uniform in $\mathbf{s}$. The first term in the minimum is obvious since $\vol (K)=1$. Fix $j \in \{1, \ldots, n-1\}$. By the change of variables $\x = k \y$ of the underlying quadratic form, the condition $ (nT_2)^{-2} a_{\mathbf{s}}^2 \preceq  k a_{\mathbf{t}}^2 k^t $ is equivalent to 
\begin{equation}\label{eq:reformulatedpreceq}
       (nT_2)^{-2} k^t a_{\mathbf{s}}^2 k \preceq  a_{\mathbf{t}}^2 = \mathrm{diag}(e^{2t_1}, \ldots, e^{2t_n}).
\end{equation}
 Denoting by $\mathbf{e}_i$ the canonical basis vectors of $\R^n$, let $U_{k,j}$ be the $n \times (n-j)$ matrix whose columns are $k\mathbf{e}_{j+1}, \ldots, k\mathbf{e}_{n}$. Then \eqref{eq:reformulatedpreceq} yields 
\begin{equation}\label{eq:neqpreceqforj}
          (nT_2)^{-2} U_{k,j}^t a_{\mathbf{s}}^2 U_{k,j} \preceq   \mathrm{diag}(e^{2t_{j+1}}, \ldots, e^{2t_n}),
\end{equation}
where $U_{k,j}^t a_{\mathbf{s}}^2 U_{k,j} $ is an $(n-j)\times (n-j)$ matrix. 
This follows from restricting the inequality $ (nT_2)^{-2} \x^t k^t a_{\mathbf{s}}^2 k \x \le \x^t a_{\mathbf{t}}^2 \x $ to the span of
$\mathbf{e}_{j+1}, \ldots, \mathbf{e}_{n}$. 
Note that since $k \in K = \SO_n(\R)$, its columns are orthonormal and thus $U_{k,j}^t U_{k,j} = I_{n-j}$.
By \eqref{eq:matrixorderimpliesdet}, the relation
\eqref{eq:neqpreceqforj} implies that
\[
    (nT_2)^{-2(n-j)} \det(U_{k,j}^ta_{\mathbf{s}}^2 U_{k,j}) \le e^{2 (t_{j+1}+\cdots+t_n) } = e^{-2(t_1 + \cdots + t_j)},
\]
since $t_1 + \cdots + t_n = 0$ in the  definition  \eqref{eq:defH+}
of $H^+$. 
Hence
\[
\int_K  \mathbf{1}_{[ (nT_2)^{-2} a_{\mathbf{s}}^2 \preceq  k a_{\mathbf{t}}^2 k^t ]} \mathrm{d}k  \le 
\int_K  \mathbf{1}_{\big[ (nT_2)^{-2(n-j)} \det(U_{k,j}^ta_{\mathbf{s}}^2 U_{k,j}) \le e^{-2(t_1 + \cdots + t_j)}\big]} \mathrm{d}k.
\]

Let us denote by
\[
    \mathbf{St}(n-j,n) = 
    \{ U \in \mathcal{M}_{n \times (n-j)}(\R): U^t U = I_{n-j}
    \},
\]
the so-called Stiefel manifold of orthonormal $(n-j)$-frames of $\R^n$.
The surjection
\[
    \pi_{n-j}: K \to \mathbf{St}(n-j,n), \quad k \mapsto U_{k,j}
\]
induces a $K$-invariant probability measure $\mu_{n-j}$ on $\mathbf{St}(n-j,n)$, by setting 
\[
    \mu_{n-j} (S) = \int_{K} \mathbf{1}_{\pi_{n-j}^{-1}(S)} \mathrm{d}k,
\]
whenever $\pi_{n-j}^{-1}(S)$ is measurable. 

We now set
\[
    S_{n-j} = \{ k \in K:~(nT_2)^{-2(n-j)} \det(U_{k,j}^ta_{\mathbf{s}}^2 U_{k,j}) \le e^{-2(t_1 + \cdots + t_j)} \}.
\]
Since $S_{n-j}$ only depends on $U_{k,j}$, we have
$ \pi_{n-j}^{-1}(\pi_{n-j}(S_{n-j})) = S_{n-j}$, whence
\[
   \int_K  \mathbf{1}_{\big[ (nT_2)^{-2(n-j)} \det(U_{k,j}^ta_{\mathbf{s}}^2 U_{k,j}) \le e^{-2(t_1 + \cdots + t_j)}\big]} \mathrm{d}k  =  \mu_{n-j} (\pi_{n-j}(S_{n-j})).
\]
We observe that 
\begin{align*}
        \pi_{n-j}(S_{n-j}) 
        &= \{ U \in \mathbf{St}(n-j,n):~\det(U^ta_\mathbf{s}^2U) \le (nT_2)^{2(n-j)} e^{-2(t_1 +\cdots+t_j)} \} \\
        &= \{ U \in \mathbf{St}(n-j,n):~\det(U^ta_\mathbf{s}^2U)^{-n/2} \ge (nT_2)^{-n(n-j)} e^{n(t_1 +\cdots+t_j)} \}.
\end{align*}
Moreover, according to Chikuse \cite[Theorem~2.2]{chikuse_macg1990}, for any symmetric positive definite matrix $A$, the function $W \mapsto \det(A)^{-(n-j)/2} \det(W^t A^{-1} W)^{-n/2}$ is a probability  density with respect to the $K$-invariant probability measure on $\mathbf{St}(n-j,n)$. In particular, this implies that
\[
    \det(A)^{-(n-j)/2}
    \int_{\mathbf{St}(n-j,n)}  \det(W^t A^{-1} W)^{-n/2}\mathrm{d}\mu_{n-j}(W) =1.
\]
Employing this with $A = a_{\mathbf{s}}^{-2}$ in combination with Markov's inequality and the fact that $\det(a_\mathbf{s})=1$ gives
\begin{align*}
        \mu_{n-j}(\pi_{n-j}(S_{n-j})) &\le (nT_2)^{n(n-j)} e^{-n(t_1 +\cdots+t_j)} \int_{\mathbf{St}(n-j,n)} \det(W^ta_\mathbf{s}^2 W)^{-n/2} \mathrm{d}\mu_{n-j}(W) \\
        &= (nT_2)^{n(n-j)} e^{-n(t_1 +\cdots+t_j)}, 
\end{align*}
which finally completes the proof of  \eqref{eq:Kintegraljbound}.

Returning to \eqref{eq:volumeEh_intermediatestep}, it follows from  \eqref{eq:Kintegraljbound} that 
\[
    \vol(E_h(T_1,T_2)) \ll \int_{H^+}  \mathbf{1}_{[t_1 \le \log(n^2T_1)]} \prod_{i<j} \sinh(t_i - t_j) e^{-n z(\mathbf{t})}\mathrm{d}\mathbf{t},
\]
where we set 
\[
    z(\mathbf{t}) = \max \left(0, \max_{1 \le j \le n-1} (t_1 + \cdots+ t_j - (n-j)\log(n T_2)) 
    \right).
\]
Moreover, we have 
\[
     \prod_{i<j} \sinh(t_i - t_j)
     = 
     \prod_{i<j} \frac{e^{t_i - t_j}-e^{t_j-t_i}}{2}
     \ll e^{\sum_{i<j}(t_i-t_j)}.
\]

For any $\mathbf{t} \in H^+$, we  claim that
\begin{align*}
        \sum_{i<j}(t_i-t_j) 
        = 2\sum_{i=1}^{n-1} \sum_{k=1}^i t_k. 
\end{align*} 
To see this, we collect the coefficients of the $t_k$ on the left hand side to get 
\[
    \sum_{i < j} (t_i - t_j) = \sum_{k=1}^{n} ((n-k)-(k-1)) t_k = \sum_{k=1}^{n} (n-2k + 1) t_k,
\]
since $t_k$ occurs positively in $n-k$ pairs $(k,j)$ with $j >k$, and negatively in the $k-1$ pairs with $i<k$. On the other hand, 
since $\sum_{k=1}^n t_k = 0$, 
we see that 
\[
2\sum_{i=1}^{n-1} \sum_{k=1}^i t_k=
    2\sum_{i=1}^n \sum_{k=1}^i t_k = 2 \sum_{k=1}^n \sum_{i=k}^n t_k = 2 \sum_{k=1}^n (n-k+1)t_k.
\]
The claim follows on subtracting $(n+1)\sum_{k=1}^n t_k = 0$ from the last expression.
Turning to the condition $\log(n^2T_1) \ge t_1$, we note that 
$\log(n^2T_1) \ge t_1 \ge t_2 \ge \cdots \ge t_n$ if  $\mathbf{t}\in H^+$, whence
\[
   t_1 + \cdots + t_i \le \min \left( 
   i \log(n^2T_1), z(\mathbf{t}) + (n-i) \log(nT_2)
   \right).
\]
Consequently, we have  
\begin{align*}
        \vol(E_h(T_1, T_2)) &\ll \int_{H^+}  \mathbf{1}_{[t_1 \le \log(n^2T_1)]} e^{2 \sum_{i=1}^{n-1} (t_1 + \cdots+ t_i)-n z(\mathbf{t})}\mathrm{d}\mathbf{t} \\
        &\ll
        \int_{H^+}  \mathbf{1}_{[t_1 \le \log(n^2T_1)]} e^{2 \sum_{i=1}^{n-1} \min \big( 
   i \log(n^2T_1), z(\mathbf{t}) + (n-i) \log(nT_2)
   \big) -n z(\mathbf{t})}\mathrm{d}\mathbf{t} .
\end{align*}
Setting $L_1 = \log(n^2T_1)$ and $L_2 = \log(nT_2)$, we will show that 
\begin{align}\label{eq:wherensquaredoverfourappears}
    2 \sum_{i=1}^{n-1} \min \left( 
   i L_1, z(\mathbf{t}) + (n-i) L_2
   \right) -n z(\mathbf{t}) \le \left\lfloor \frac{n^2}{4}\right\rfloor ( L_1 + L_2 ).
\end{align}
This will yield 
\[
    \vol(E_h(T_1, T_2)) \ll e^{ \left\lfloor n^2/4 \right\rfloor ( L_1 + L_2 ) } 
    \int_{H^+} \mathbf{1}_{[t_1 \le L_1]} \mathrm{d}\mathbf{t} \ll_{\varepsilon} (T_1  T_2)^{ \left\lfloor n^2/4 \right\rfloor + \varepsilon },
\]
since the volume of the remaining truncated cone is 
\[
    \vol \{ \mathbf{t} \in H^+: t_1 \le L_1 \} \ll (1+ L_1)^{n-1} \ll_{\varepsilon} T_1^\varepsilon.
\]

In  order to finish our proof, it  remains to establish \eqref{eq:wherensquaredoverfourappears}. We distinguish the cases of even and odd $n$.
Let us first assume that $n=2m+1$ is odd. Then 
\begin{align*}
    2 \sum_{i=1}^{n-1} \min \left( 
   i L_1, z(\mathbf{t}) + (n-i) L_2
   \right) -n z(\mathbf{t}) &\le 2  \sum_{i=1}^m i L_1 + 2\!\!\sum_{i=m+1}^{2m} \!(z(\mathbf{t}) + (n-i) L_2)  -n z(\mathbf{t}) \\
   &= m(m+1)(L_1 + L_2) - z(\mathbf{t})
   \\
   &\le m(m+1)(L_1 + L_2) 
   \\
   & = \Big\lfloor \frac{n^2}{4} \Big\rfloor (L_1 + L_2).
\end{align*}
Now, let us assume that $n = 2m$ is even. If $L_1 \ge L_2$, then 
\begin{align*}
      2 \sum_{i=1}^{n-1} \min \left( 
   i L_1, z(\mathbf{t}) + (n-i) L_2
   \right) -n z(\mathbf{t}) &\le 2 \sum_{i=1}^{m-1} i L_1 + 2\!\sum_{i=m}^{2m-1} (z(\mathbf{t}) + (n-i) L_2)  -n z(\mathbf{t}) \\
   &= m(m-1) L_1 + m(m+1)L_2 \\
   &\le m^2 (L_1 + L_2)\\
   &= \Big\lfloor \frac{n^2}{4} \Big\rfloor (L_1 + L_2).
\end{align*}
If $L_2 \ge L_1$, we instead have 
\begin{align*}
      2 \sum_{i=1}^{n-1} \min \left( 
   i L_1, z(\mathbf{t}) + (n-i) L_2
   \right) -n z(\mathbf{t}) &\le 2 \sum_{i=1}^{m} i L_1 + 2\!\!\sum_{i=m+1}^{2m-1} \!(z(\mathbf{t}) + (n-i) L_2)  -n z(\mathbf{t}) \\
   &= m(m+1) L_1 + m(m-1)L_2 -2z(\mathbf{t}) \\
   &\le m^2 (L_1 + L_2)\\
   &= \Big\lfloor \frac{n^2}{4} \Big\rfloor (L_1 + L_2). \qedhere
\end{align*}
\end{proof}

\begin{proof}[Proof of Theorem~\ref{thm:uniformlatticecount}]
    Let us fix $g,h \in \GL_n(\R)$ with $|\det(g)|=|\det(h)|=1$. On possibly replacing $(g,h)$ by $ (\mathrm{diag}(-1,1,\ldots,1)g,  \mathrm{diag}(-1,1,\ldots,1)h) $, which leaves the relevant inequalities unchanged, we may assume $\det(g)=1$.
   We  choose a compact neighbourhood  $\mathcal{O} \subset \SL_n(\R)$ of the identity,
    with the property that the  translates $U \mathcal{O}$, $U \in \SL_n(\Z)$, are pairwise disjoint.  Concretely, we may take
\begin{align*}
    \mathcal{O} = 
    \left\{ 
    o \in \SL_n(\R):
    | 
    I_n - o
    | \le 1/(2n+2), \ 
    | 
    I_n - o^{-1}
    | \le 1/(2n+2)
    \right\}.
\end{align*}
We claim that  the translates 
$U \mathcal{O}$ are pairwise disjoint
for $U \in \SL_n(\Z)$.
Indeed, 
assume that there are 
$U_1, U_2 \in \SL_n(\Z)$ and $o_1, o_2 \in \mathcal{O}$ with $U_1 o_1 = U_2 o_2$. Then we have 
\[
   o_2 o_1^{-1} =  U_2^{-1} U_1 \in \SL_n(\Z).
\]
On the other hand, we have  $|o_2|\le |o_2 - I_n|+|I_n| \le 2$ and
\[
    |I_n - o_2 o_1^{-1}| \le |I_n - o_2| + |o_2 - o_2 o_1^{-1}| \le \frac{1}{2n+2} + n|o_2||I_n -o_1^{-1}| \le \frac{2n+1}{2n+2} < 1.
\]
Since $o_2 o_1^{-1} \in \SL_n(\Z)$, the matrix $I_n - o_2o_1^{-1}$ is an integer matrix with 
 norm  strictly less than $1$. Thus  $o_1 = o_2$, which implies that  $U_1 = U_2$. 

We clearly have  $|o|,|o^{-1}| \le 2$ for all $o \in \mathcal{O}$. For $U \in \SL_n(\Z)$ with $gU \in E_h(T_1, T_2)$ and $o \in \mathcal{O}$, we have  
\[
    |gUo| \le n |gU||o| \le 2n|gU|, 
\]
and 
\[
    |(gUo)^{-1}h| = |o^{-1}(gU)^{-1}h| \le n |o^{-1}||(gU)^{-1}h| \le2 n |(gU)^{-1}h|.
\]
Thus  it follows that 
\[
    gU \mathcal{O} \subseteq E_{h}(2nT_1, 2nT_2).
\]
Since the sets $U \mathcal{O}$, $U \in \SL_n(\Z)$, are disjoint, so are their translates $gU \mathcal{O}$, $U \in \SL_n(\Z)$. 
Therefore we get 
\[
    \vol(\mathcal{O}) \#\{ U \in \SL_n(\Z):~ gU \in E_{h}(T_1,T_2)\} \le \vol(E_{h}(2nT_1,2nT_2)). 
\]
Since $\vol(\mathcal{O})\gg 1$, Lemma~\ref{lem:volumeuniformbound} yields  
\[
    \#\{ U \in \SL_n(\Z):~ gU \in E_{h}(T_1,T_2)\} \ll_{\varepsilon} (T_1 T_2)^{\lfloor n^2/4 \rfloor+\varepsilon} ,
\]
which clearly completes the proof.
\end{proof}

\subsection{Conclusion} \label{section:uniform_proof}

We now employ the results from Section~\ref{section:intermediate_lattices} and Section~\ref{section:boundingnumberofbases} to prove the desired uniform upper bound.

\begin{proof}[Proof of Theorem~\ref{thm:uniform}]
    Let $M\in \mathcal{M}_n(\Z)$ be non-singular.
    Our starting point is \eqref{eq:chicken}, which gives
$$
        \tau_n(T,M) 
    =
    2
    \sum_{M\Z^n \subseteq \mathsf{\Lambda} \subseteq \Z^n} \# \{ 
    U \in \SL_n(\Z):  | B_\mathsf{\Lambda} U | \le T, \: | U^{-1}B_\mathsf{\Lambda}^{-1}M|\le T
    \},
    $$
where
$B_\mathsf{\Lambda}$ is any basis matrix 
associated to the lattice  $\mathsf{\Lambda}$.
Putting  $g_\mathsf{\Lambda} = \det(\mathsf{\Lambda})^{-1/n} B_\mathsf{\Lambda}$ and $h_M = |\det(M)|^{-1/n}M$, we obtain
    \begin{align*}
        \tau_n(T,M) 
    =
    2
    \sum_{M\Z^n \subseteq \mathsf{\Lambda} \subseteq \Z^n} \!\!\! \# \left\{ 
    U \in \SL_n(\Z):  | g_\mathsf{\Lambda} U | \le T_1, \: | (g_\mathsf{\Lambda}U)^{-1}h_M|\le 
    T_2
    \right\},
    \end{align*}
where
$$
T_1=\frac{T}{\det(\mathsf{\Lambda})^{1/n}} , \quad 
T_2=\frac{T\det(\mathsf{\Lambda})^{1/n}}{|\det(M)|^{1/n}}.
$$
Note that  $|\det(g_\mathsf{\Lambda}) |=|\det(h_M)|=1$. Thus  Theorem~\ref{thm:uniformlatticecount} can be applied to bound each summand. 
Applying Proposition~\ref{prop:latticeinclusion}, it now follows that 
    \begin{align*}
          \tau_n(T,M) 
    &\ll_{\varepsilon} 
     \sum_{M\Z^n \subseteq \mathsf{\Lambda} \subseteq \Z^n} T^{2 \lfloor n^2 /4 \rfloor + \varepsilon} |\det(M)|^{- \lfloor n^2 /4 \rfloor/n - \varepsilon/(2n)} \ll_{\varepsilon }  
      T^{ \lfloor n^2 /2 \rfloor + \varepsilon},
    \end{align*}
since 
    $2 \lfloor n^2 /4 \rfloor =  \lfloor n^2 /2 \rfloor$.
\end{proof}

\section{Factorisations of the zero matrix}\label{sec:On}

In this section we prove Theorem \ref{thm:factoring_zero} using an argument inspired by that found in work of 
Katznelson \cite{katznelson}.
Katznelson views the set of singular matrices as the union
\begin{align}\label{eq:katznelsonunion}
    \bigcup_{\boldsymbol{{\lambda}} \in \ZZp^n} \{ A \in \mathcal{M}_n(\Z): A\boldsymbol{{\lambda}} = \mathbf{0}  \},
\end{align}
observing that the relation $A\boldsymbol{{\lambda}} = \mathbf{0}$ means that all the rows of $A$ lie inside the primitive lattice of rank $n-1$ that is orthogonal to $\boldsymbol{\lambda}$. In order to find the number of such matrices inside an expanding convex body, he uses lattice point counting arguments and handles the error produced by the fact that the union in \eqref{eq:katznelsonunion} is not a partition.

In our setting, we observe that the rows of every matrix $A \in \mathcal{M}_n(\Z)$ lie inside a unique primitive $\rank(A)$-dimensional lattice $\Lambda$. For any matrix $B \in \mathcal{M}_n(\Z)$, we have $AB = O_n$ if and only if the columns of $B$ lie inside the orthogonal lattice $\Lambda^\perp$. This allows us to partition the set of solutions $(A,B)$ to the equation $AB = O_n$ as 
\begin{align}\label{eq:partitionlatticesAB}
\bigsqcup_{r=0}^n
    \bigsqcup_{\substack{\Lambda \subset \Z^n \text{ prim.} \\
    \rank(\Lambda)=r}}  
\mathcal{A}_{\Lambda} \times
\mathcal{B}_{\Lambda^\perp},
\end{align}
where we write $\mathcal{A}_{\Lambda} = \{ A \in \mathcal{M}_n(\Z): \rank(A)=\rank(\Lambda),~\text{rows of $A$ lie in $\Lambda$}\} $ and $ \mathcal{B}_{\Lambda^\perp} =
    \{ B \in \mathcal{M}_n(\Z): ~\text{columns of $B$ lie in $\Lambda^\perp$}\}$.
    We then employ geometry of numbers techniques in order to
    count the number of matrices of bounded norm inside $\mathcal{A}_{\Lambda}$ and $\mathcal{B}_{\Lambda^\perp}$, which eventually leads to the asymptotic given in Theorem~\ref{thm:factoring_zero}.

\subsection{Geometry of numbers}

We start by collecting various definitions and lemmas from the geometry of numbers that are needed for our proof. 
Let $n \ge 1$ be a positive integer and let $r \in \{0,\ldots,n\}$. 
For $R \in \{\Z, \Q, \R \}$ and any set $\mathcal{S} \subset \Z^n$, we write 
\[
    \langle \mathcal{S} \rangle_R = \left\{ \sum_{i=1}^m r_is_i: m\ge0,~ r_1,\ldots,r_m \in R,~s_1,\ldots,s_m \in \mathcal{S} \right\}
\]
for the span of $\mathcal{S}$ over $R$. 
We denote by $\mathcal{L}_{n,r}$ the set of all lattices of rank $r$ in $\Z^n$, and we denote by $\mathcal{P}_{n,r}$ the set of all {\it primitive} lattices of rank $r$ in $\Z^n$; i.e.~the set of  $\Lambda \in \mathcal{L}_{n,r}$ with $\Lambda = \langle \Lambda \rangle_{\Q} \cap \Z^n$. 
For $k \le r$, the $k$-th {\it successive minimum} of $\Lambda$ is the minimal radius $\lambda_k$ such that the closed ball of radius $\lambda_k$ centered at the origin contains $k$ linearly independent points of $\Lambda$. By Minkowski's second theorem, we have 
\begin{align}\label{eq:minkowsi2thm}
        \lambda_1 \cdots \lambda_r \ll \det(\Lambda)
    \ll \lambda_1 \cdots \lambda_r,
\end{align}
where the implicit constants only depend on $n$.

 For a lattice $\Lambda \in \mathcal{P}_{n,r}$, we denote by $\Lambda^\perp$ the {\it orthogonal lattice} of 
$\Lambda$; we recall that $\Lambda^\perp$ may be obtained as $\Lambda^\perp = \Z^n \cap V^\perp$, where $V^\perp$ is the subspace of $\Q^n$ orthogonal to the subspace $V = \langle \Lambda \rangle_{\Q}$. The following lemma is classical and is proved in Schmidt \cite[\S 2]{Schm}, for example. 

\begin{lem}\label{lem:detlatticeorthogonal}
        Let $\Lambda \subset \Z^n$ be a primitive lattice. Then 
        $
        \det(\Lambda) = \det(\Lambda^\perp).
        $
    \end{lem}

Next, for $\Lambda \in \mathcal{P}_{n,r}$, we define 
     \begin{align}
    \alpha_{n,r}(T,\Lambda) &= 
        \# 
    \{
    (\a_1, \ldots, \a_n) \in (\Lambda \cap [-T,T]^n)^n: \rank (\langle \a_1, \ldots, \a_n\rangle_{\Z}) = r
    \},
    \label{eq:defalpha}
    \\ \beta_{n,r}(T,\Lambda)
    &=
    \# \{ 
    (\b_1, \ldots, \b_n) \in (\Lambda \cap [-T,T]^n)^n
    \},
    \label{eq:defbeta}
    \end{align} and
    \begin{align} \label{eq:defgamma}
    \gamma_{n,r}(T,\Lambda) =  \# 
    \{
    (\a_1, \ldots, \a_n) \in (\Lambda \cap [-T,T]^n)^n: \rank (\langle \a_1, \ldots, \a_n\rangle_{\Z}) < r
    \}.
    \end{align}
    With $\mathcal{A}_{\Lambda}$ and $\mathcal{B}_{\Lambda^\perp}$ defined as in \eqref{eq:partitionlatticesAB}, we remark that $\alpha_{n,r}(T,\Lambda) = \# (\mathcal{A}_{\Lambda}\cap [-T,T]^{n^2})$ and $\beta_{n,n-r}(T,\Lambda^\perp) = \# (\mathcal{B}_{\Lambda^\perp} \cap [-T,T]^{n^2})$, where $r = \rank (\Lambda)$. 
       \begin{lem}\label{lem:condition_alphanonzero_detsize}
        Let $n \ge 2$, $r \in \{1, \ldots, n-1\}$ and $\Lambda \in \mathcal{P}_{n,r}$. 
        Then there exist constants $q_n,Q_n>0$ such that the following hold:
        \begin{enumerate}
            \item if $\alpha_{n,r}(T, \Lambda) \neq 0$, then we have $\det(\Lambda) \le Q_n T^{r}$;
            \item if 
            $\alpha_{n,r}(T, \Lambda) = 0$, then we have $\det(\Lambda) \ge q_nT$. 
        \end{enumerate}
    \end{lem}
    \begin{proof}
       If $\alpha_{n,r}(T, \Lambda) \neq 0$, then $\Lambda$ contains $r$ linearly independent vectors $\a_{i_1}, \ldots, \a_{i_r}$ with $|\a_{i_1}|\le \cdots \le |\a_{i_r}|\le T$. This means that the successive minima $\lambda_1, \ldots, \lambda_r$ of $\Lambda$  satisfy $\lambda_k \ll |\a_{i_k}| \le T$, and thus, by \eqref{eq:minkowsi2thm} we have 
    $
    \det (\Lambda) \ll T^r. 
    $
    
    Conversely, if $\alpha_{n,r}(T,\Lambda) = 0$, then $\Lambda \cap [-T,T]^n$ does not contain $r$ linearly independent vectors.
    This implies that $\lambda_r \ge T$,
    and hence \eqref{eq:minkowsi2thm} yields $\det(\Lambda) \gg T$. 
    \end{proof}

    Since there are only finitely many lattices of bounded determinant in $\mathcal{P}_{n,r}$, Lemma~\ref{lem:condition_alphanonzero_detsize} shows that there are only finitely many $\Lambda$ with 
    $\alpha_{n,r}(T, \Lambda) \neq 0$. In fact, Schmidt \cite[Theorem~1]{Schm} proved the following asymptotic formula. 
    \begin{lem}
    \label{lem:nooflatticesofboundedheight}
        Let $n \ge 2$ and $ r \in \{ 1, \ldots, n-1 \}$. Then
        \[
        \# \{ \Lambda \in \mathcal{P}_{n,r}: \det(\Lambda) \le H \}
        \sim a_{n,r} H^n,
        \]
 as $H \to \infty$, 
        where $a_{n,r}$ is a positive constant.
    \end{lem}

    For an $r$-dimensional subspace $V \subset \R^n$,
    we denote by $\mu_V$
    the Haar measure on $V$ that is induced by the Euclidean structure on $\R^n$. This means that the measure is normalised in such a way that 
    \begin{align}\label{eq:unitballvolume}
         \mu_V(\{v \in V: \|v \|_2 \le 1 \}) = B_r,
    \end{align}
    where $B_r = \pi^{r/2}/\Gamma(\tfrac{r}{2}+1)$ is the usual volume of the $r$-dimensional unit ball and $\|v\|_2$ denotes the Euclidean norm on $\R^n$.

 The following lemma contains the lattice point counting estimates we require and appears in many variants in the literature, such as in work of Davenport \cite{davenport_on_a_principle},  Schmidt \cite[Lemma~2]{Schm} and   Katznelson  \cite[Lemma~2]{katznelson}. Our formulation follows from work of Widmer \cite[Theorem~5.4]{widmer}. 
 
    \begin{lem}\label{lem:generalboundlatticesuccessiveminima}
          Let $n \ge 2$, let $r \in \{ 1, \ldots, n-1\}$, and let $\Lambda \in \mathcal{L}_{n,r}$ be a lattice with successive minima $\lambda_1, \ldots, \lambda_r$. 
          Then 
              \begin{equation*}
     \left|
            \# (\Lambda \cap [-T,T]^n) - \frac{\nu_\Lambda}{\det(\Lambda)} T^{r}
            \right| \ll  1+ \sum_{i=1}^{r-1} \frac{T^{i}}{\lambda_1 \cdots \lambda_{i}},
              \end{equation*}
           as $T \to \infty$, 
          where
                  the implicit constant depends only on $n$  and
\begin{equation}\label{eq:villach}
\nu_\Lambda = \mu_{  \langle \Lambda\rangle_\R}(
        \langle \Lambda\rangle_\R \cap [-1,1]^n).
\end{equation}
       \end{lem}

We record the following immediate consequence.

       \begin{cor}\label{lem:asymp_for_beta}
       Let $n \ge 2$, let $r \in \{ 1, \ldots, n-1\}$, let $\Lambda \in \mathcal{L}_{n,r}$ be a lattice with successive minima $\lambda_1, \ldots, \lambda_r$ and let $\nu_{\Lambda}$ be given by \eqref{eq:villach}. 
          \begin{enumerate}
              \item We have
              \begin{align}\label{eq:latticeptsinsideregion}
          \left|
            \# (\Lambda \cap [-T,T]^n) - \frac{\nu_\Lambda}{\det(\Lambda)} T^{r}
            \right| \ll T^{r-1} ,
          \end{align}
          where
          the implicit constant depends on $n$ only.
          \item Let $w_n$ be a fixed constant.  If the successive minima of $\Lambda$ satisfy the bound $\lambda_1  \le \cdots \le \lambda_r \le w_nT$, we have  
        \begin{align}\label{eq:boundedlatticecount}
            \left|
            \# (\Lambda \cap [-T,T]^n) - \frac{\nu_\Lambda}{\det(\Lambda)} T^{r}
            \right| \ll \frac{T^{r-1} \lambda_r}{\det(\Lambda)} ,
        \end{align}
where the implicit constant  depends on $n$ and $w_n$ only.
          \end{enumerate}       
    \end{cor}
\begin{proof}
    The successive minima of any integer lattice  are the Euclidean norms of some integer vectors, and therefore we have $\lambda_i \ge 1$ for all $i =1, \ldots, r$. In combination with Lemma~\ref{lem:generalboundlatticesuccessiveminima}, this immediately implies  \eqref{eq:latticeptsinsideregion}.

    Let us now assume that $\lambda_1 \le \cdots \le \lambda_r \le w_nT$.
   By  \eqref{eq:minkowsi2thm}, we have 
   \[
   \frac{T^i}{\lambda_1 \cdots \lambda_i} = \frac{T^i\lambda_{i+1} \cdots \lambda_r}{\lambda_1 \cdots \lambda_r} \ll \frac{T^{r-1}\lambda_r}{\det(\Lambda)}.
   \]
The bound  \eqref{eq:boundedlatticecount}
now follows from
 Lemma~\ref{lem:generalboundlatticesuccessiveminima}.
\end{proof}

In the following proofs we use the fact that $\nu_{\Lambda}$ can be bounded by a constant  only depending on $n$. Indeed, 
   it follows from  \eqref{eq:unitballvolume} and \eqref{eq:villach} that 
   \[
   \nu_{\Lambda} = \mu_{  \langle \Lambda\rangle_\R} 
   (
   \{ 
   v \in  \langle \Lambda\rangle_\R : |v| \le 1 
   \}) \le 
   \mu_{  \langle \Lambda\rangle_\R} 
   (
   \{ 
   v \in  \langle \Lambda\rangle_\R : \|v\|_2 \le \sqrt{n} 
   \}) = n^{r/2}B_r \ll 1.
   \]
   We now prove the following result. 

\begin{lem}\label{lem:gammabound}
          Let $n \ge 2$, let $r \in \{ 1, \ldots, n-1\}$, let $q_n$ be a fixed constant and let 
     $\Lambda \in \mathcal{P}_{n,r}$.
Assume that the successive minima of $\Lambda$ satisfy  
$\lambda_1  \le \cdots \le \lambda_r \le q_nT$. Then 
        \[
        \gamma_{n,r}(T, \Lambda) \ll \frac{T^{nr-1} \lambda_r}{\det(\Lambda)^n},
        \]
        where $\gamma_{n,r}(T, \Lambda)$ is defined in \eqref{eq:defgamma}. 
        The implicit constant only depends on $n$ and $q_n$. 
    \end{lem}

     \begin{proof}
        Let $\a_1, \ldots, \a_n \in \Lambda \cap [-T,T]^n$ with the property that \[
        \rank(\langle \a_1, \ldots, \a_n \rangle_{\Z}) = s < r,
        \] 
        for some $s\in \{0,\ldots,r-1\}$.
        Assume, say, that $\rank(\langle \a_1, \ldots, \a_s \rangle_{\Z}) = s$. 
         Combining Corollary~\ref{lem:asymp_for_beta} with the bound  $\nu_\Lambda \ll 1$ and the assumption $T/\lambda_r \gg 1$, we get
        \begin{align*}
            \# (\Lambda \cap [-T,T]^n) 
        \ll \frac{T^r}{\det (\Lambda)} +  \frac{T^{r-1}\lambda_r}{\det (\Lambda)}
        \ll  \frac{T^r}{\det (\Lambda)}.
        \end{align*} 
        Thus, there are 
        \[
        \ll  \left( \frac{T^r}{\det (\Lambda)}\right)^s
        \]
        choices for $\a_1,\ldots,\a_s$.
        
        We have 
        \[
        \a_{s+1}, \ldots, \a_n \in 
        \Lambda' := 
        \langle \a_1, \ldots, \a_s\rangle_{\Q} \cap \Lambda.
        \]
        Let $\lambda_1', \ldots, \lambda_s'$ denote the successive minima of $\Lambda'$. Since $\Lambda' \subset \Lambda$, we have $\lambda_i \le \lambda_i'$ for $i = 1,\ldots,s$.        
       Observing that, since $\a_1, \ldots, \a_s \in \Lambda'$, we have $\lambda_1' \le \cdots \le \lambda_s' \ll T$, and employing the bound $\lambda_1'\cdots\lambda_s' \ll \det(\Lambda') \ll \lambda_1'\cdots\lambda_s'$ from \eqref{eq:minkowsi2thm} and Corollary~\ref{lem:asymp_for_beta}, we conclude that
    \[
     \# (\Lambda ' \cap [-T,T]^n) 
     \ll \frac{T^s}{\det(\Lambda')}
         \ll  \frac{T^s}{\lambda_1'\cdots\lambda_s'} \ll \frac{T^s}{\lambda_1\cdots\lambda_s}.
    \] 
    This shows that there are 
    \[
    \ll \left(
    \frac{T^s}{\lambda_1\cdots\lambda_s}
    \right)^{n-s}
    \]
    choices for $\a_{s+1},\ldots,\a_n$.
    It therefore  follows that 
    \begin{align}\label{eq:gammasumbound}
    \gamma_{n,r}(T, \Lambda) 
    \ll
    \sum_{s=0}^{r-1} \left(\frac{T^r}{\det(\Lambda)}
    \right)^s 
    \left(\frac{T^s}{\lambda_1\cdots\lambda_s}
    \right)^{n-s}.
    \end{align}
Moreover, since  $\lambda_i \ll T$, it follows from \eqref{eq:minkowsi2thm} that 
    \[
    \frac{T^s}{\lambda_1 \cdots \lambda_s} \ll \frac{T^r}{\lambda_1 \cdots \lambda_r  } \ll 
    \frac{T^r}{\det(\Lambda)},
    \]
    and so the right hand side of \eqref{eq:gammasumbound} is bounded from above by the contribution from $s=r-1$. Thus
    \begin{align*}
        \gamma_{n,r}(T, \Lambda) \ll \frac{T^{r(r-1)+(r-1)(n-r+1)}}{\det(\Lambda)^{r-1}(\lambda_1 \cdots \lambda_{r-1})^{n-r+1}} 
        \ll 
        \frac{T^{rn+r-n-1}\lambda_r^{n-r+1}}{\det(\Lambda)^{n}} \ll \frac{T^{rn-1}\lambda_r}{\det(\Lambda)^{n}}.
    \end{align*}
    For the second bound, we employed  \eqref{eq:minkowsi2thm} to obtain $\lambda_1 \cdots \lambda_{r-1} \gg \det(\Lambda)/\lambda_r$; for the final bound, we used the fact that $\lambda_r\ll T$.    
    \end{proof}
    
\begin{lem}
\label{lem:inverse_determinant_bound}
Let $ n \ge 2$ and $r\in \{1,\ldots,n-1 \}$. 
For every integer $s \ge 0$, we have
\begin{equation*}
    \sum_{\substack{\Lambda \in \mathcal{P}_{n,r}, \\
    \alpha_{n,r}(T,\Lambda) \neq 0}}
    \frac{1}{\det(\Lambda)^s}
    \ll
    \begin{cases}
        T^{r(n-s)} & \text{if } s < n,\\
        \log T & \text{if } s = n,\\
        1 & \text{if } s > n,
    \end{cases}
\end{equation*}
 as $T \to \infty$. The implicit constant depends on $n$ only.
\end{lem}

\begin{proof}
By Lemma~\ref{lem:condition_alphanonzero_detsize}, we may weaken the condition
$\alpha_{n,r}(T,\Lambda)\ne0$ to the condition
$\det(\Lambda)\le Q_nT^r$, obtaining
\begin{equation*}
    \sum_{\substack{\Lambda \in \mathcal{P}_{n,r}, \\
    \alpha_{n,r}(T,\Lambda) \neq 0}}
    \frac{1}{\det(\Lambda)^s}
    \le
    \sum_{\substack{\Lambda \in \mathcal{P}_{n,r}, \\
    \det(\Lambda) \le Q_n T^r}}
    \frac{1}{\det(\Lambda)^s}.
\end{equation*}
To bound the latter sum, we use a dyadic decomposition, splitting the
range of possible determinants $[1,Q_nT^r]$ into logarithmically many
intervals:
\begin{equation*}
    \sum_{\substack{\Lambda \in \mathcal{P}_{n,r}, \\
    \alpha_{n,r}(T,\Lambda) \neq 0}}
    \frac{1}{\det(\Lambda)^s}
    \le
    \sum_{k = 0}^{\lceil \log_2(Q_n T^r) \rceil}\!\!\!
    \sum_{\substack{\Lambda \in \mathcal{P}_{n,r}, \\
    2^k \le \det(\Lambda) < 2^{k + 1}}}\!\!\!
    \frac{1}{\det(\Lambda)^s}.
\end{equation*}
We now apply Lemma~\ref{lem:nooflatticesofboundedheight}. In particular,
the number of lattices satisfying $\det(\Lambda)\le 2^{k + 1}$ is
$\ll 2^{(k + 1)n}$. Hence we get
\begin{equation*}
      \sum_{k = 0}^{\lceil \log_2(Q_n T^r) \rceil}\!\!\!
    \sum_{\substack{\Lambda \in \mathcal{P}_{n,r}, \\
    2^k \le \det(\Lambda) < 2^{k + 1}}}\!\!\!
    \frac{1}{\det(\Lambda)^s}
    \ll
    \sum_{k = 0}^{\lceil \log_2(Q_n T^r) \rceil}
    \frac{2^{(k + 1)n}}{2^{ks}}.
\end{equation*}
The sum on the right is a geometric sum. It is dominated by its final
term when $s<n$, has logarithmically many terms of bounded size when
$s=n$, and is bounded by a convergent geometric series when $s>n$.
Therefore,
\begin{equation*}
   \sum_{k = 0}^{\lceil \log_2(Q_n T^r) \rceil}
    \frac{2^{(k + 1)n}}{2^{ks}}
    =
    2^{n}
    \sum_{k = 0}^{\lceil \log_2(Q_n T^r) \rceil}
    2^{k(n-s)}
    \ll
    \begin{cases}
        T^{r(n-s)} & \text{if } s < n,\\
        \log T & \text{if } s = n,\\
        1 & \text{if } s > n,
    \end{cases}
\end{equation*}
which concludes the proof.
\end{proof}

\begin{lem}\label{lem:gamma_beta_sum_bound}
    Let $n \ge 2$ and let $r \in \{1, \ldots, n-1\}$. We have
    \[
    \sum_{\substack{\Lambda \in \mathcal{P}_{n,r}, \\
    \alpha_{n,r}(T,\Lambda) \neq 0}}
    \gamma_{n,r}(T,\Lambda) \beta_{n,n-r}(T,\Lambda^\perp) \ll T^{n^2 - 1}.
    \]
\end{lem}
\begin{proof}
By the definition of $\beta_{n,n-r}(T, \Lambda^\perp)$ in \eqref{eq:defbeta}, we have
\begin{equation*}
    \beta_{n,n-r}(T, \Lambda^\perp)  = \left(\#(\Lambda^\perp \cap [-T,T]^n)\right)^n.
\end{equation*}
Using the estimate for $\#(\Lambda^\perp \cap [-T,T]^n)$ from Corollary~\ref{lem:asymp_for_beta}, we obtain
\begin{equation*}
    \beta_{n,n-r}(T, \Lambda^\perp) = \left(\frac{\nu_{\Lambda^\perp}}{\det(\Lambda^\perp)}T^{n-r} + O(T^{n-r - 1})\right)^n \ll \frac{\nu_{\Lambda^\perp}^n}{\det(\Lambda^\perp)^n}T^{n(n-r)} + T^{n^2 -n(r+1)}.
\end{equation*}
Combining this with the upper bound for $\gamma_{n,r}(T,\Lambda)$ from Lemma~\ref{lem:gammabound}, we obtain
\begin{equation*}
     \sum_{\substack{\Lambda \in \mathcal{P}_{n,r}, \\
        \alpha_{n,r}(T,\Lambda) \neq 0}}
        \gamma_{n,r}(T,\Lambda) \beta_{n,n-r}(T,\Lambda^\perp) \ll \sum_{\substack{\Lambda \in \mathcal{P}_{n,r}, \\
        \alpha_{n,r}(T,\Lambda) \neq 0}} \frac{T^{nr - 1}\lambda_r}{\det(\Lambda)^{n}} \left(\frac{\nu_{\Lambda^{\perp}}^nT^{n(n - r)}}{\det(\Lambda^{\perp})^n} + T^{n^2 - n(r + 1)}\right).
\end{equation*}
Clearly  \eqref{eq:minkowsi2thm} implies that $\lambda_r \ll \det (\Lambda)$ and 
 Lemma~\ref{lem:detlatticeorthogonal} yields $\det (\Lambda) = \det (\Lambda^{\perp})$.  Substituting these bounds gives
\begin{equation*}
     \sum_{\substack{\Lambda \in \mathcal{P}_{n,r}, \\
        \alpha_{n,r}(T,\Lambda) \neq 0}}
        \gamma_{n,r}(T,\Lambda) \beta_{n,n-r}(T,\Lambda^\perp) \ll \sum_{\substack{\Lambda \in \mathcal{P}_{n,r}, \\
        \alpha_{n,r}(T,\Lambda) \neq 0}} \frac{T^{nr - 1}}{\det(\Lambda)^{n - 1}} \left(\frac{T^{n(n - r)}}{\det(\Lambda)^n} + T^{n^2 - n(r + 1)}\right).
\end{equation*}
Finally, we apply Lemma~\ref{lem:inverse_determinant_bound} with exponents $s=2n-1>n$ and $s=n-1<n$, respectively. This gives 
\begin{equation*} 
\sum_{\substack{\Lambda \in \mathcal{P}_{n,r}, \\ \alpha_{n,r}(T,\Lambda) \neq 0}} \gamma_{n,r}(T,\Lambda) \beta_{n,n-r}(T,\Lambda^\perp) \ll T^{n^2-1} +T^{n^2-n-1}T^r \ll T^{n^2-1},
\end{equation*} 
where the final inequality follows from $r\leq n-1$. 
\end{proof}

\begin{lem}\label{lem:beta_beta_sum_asymp}
    Let $n \ge 2$ and let $r \in \{1,\ldots,n-1\}$. Then
    \[
    \sum_{\substack{\Lambda\in\mathcal{P}_{n,r}\\
    \alpha_{n,r}(T,\Lambda)\neq0}}
    \beta_{n,r}(T,\Lambda)
    \beta_{n,n-r}(T,\Lambda^{\perp})
    =
    c_{O_n,r}T^{n^2}+O(T^{n^2-1}),
    \]
    where
    \[
    c_{O_n,r}
    =
    \sum_{\substack{\Lambda\in\mathcal{P}_{n,r}}}
    \frac{\nu_{\Lambda}^n\nu_{\Lambda^\perp}^n}
    {\det(\Lambda)^{2n}}
    \]
    is a positive constant. The implicit constant depends on $n$ only.
\end{lem}

\begin{proof}
By the definition of $\beta_{n,r}(T,\Lambda)$, we have
\begin{equation*}
    \beta_{n,r}(T,\Lambda)
    =
    \#\{(\b_1,\ldots,\b_n)\in(\Lambda\cap[-T,T]^n)^n\}
    =
    \left(\#(\Lambda\cap[-T,T]^{n})\right)^n.
\end{equation*}
Suppose that $\alpha_{n,r}(T,\Lambda)\neq0$. Then $\Lambda$ contains
$r$ linearly independent vectors in $[-T,T]^n$. Consequently, its
successive minima satisfy
\[
\lambda_1\leq\cdots\leq\lambda_r\leq\sqrt n\,T,
\]
so the second estimate \eqref{eq:boundedlatticecount} in Corollary~\ref{lem:asymp_for_beta} applies.
Since $\lambda_r \ll \min (\det(\Lambda),T)$, we have 
\[
    \frac{T^{r-1}\lambda_r}{\det(\Lambda)} \ll \min \left( T^{r-1}, \frac{T^r}{\det(\Lambda)}
    \right) .
\]
Hence we get 
\begin{align*}
 \beta_{n,r}(T,\Lambda)
 =
 \left(
 \frac{\nu_{\Lambda}T^r}{\det(\Lambda)}
 +
 O\left(
 \frac{T^{r-1}\lambda_r}{\det(\Lambda)}
 \right)
 \right)^n
 =
 \frac{\nu_{\Lambda}^nT^{rn}}{\det(\Lambda)^n}
 +
 O\left(
 \frac{T^{rn-1}}{\det(\Lambda)^{n-1}}
 \right).
\end{align*}

For $\beta_{n,n-r}(T,\Lambda^\perp)$, we use the first estimate \eqref{eq:latticeptsinsideregion} in
Corollary~\ref{lem:asymp_for_beta}. Since
$\nu_{\Lambda^\perp}\ll 1$ and 
$\det(\Lambda^\perp)=\det(\Lambda)$,
by Lemma \ref{lem:detlatticeorthogonal},
we obtain
\begin{align*}
 \beta_{n,n-r}(T,\Lambda^\perp)
 &=
 \left(
 \frac{\nu_{\Lambda^\perp}T^{n-r}}{\det(\Lambda^\perp)}
 +O(T^{n-r-1})
 \right)^n\\
 &=
 \frac{\nu_{\Lambda^\perp}^nT^{n(n-r)}}{\det(\Lambda)^n}
 +
 O\left(
 T^{n(n-r-1)} + \frac{T^{n^2-nr-1}}{\det(\Lambda)^{n-1}}
 \right).
\end{align*}

Multiplying the two estimates gives
\begin{align*}
 \beta_{n,r}(T,\Lambda)
 \beta_{n,n-r}(T,\Lambda^\perp)
 =
 \frac{\nu_{\Lambda}^n\nu_{\Lambda^\perp}^nT^{n^2}}
 {\det(\Lambda)^{2n}}
 &+ O\left(
 \frac{T^{n^2-1}}{\det(\Lambda)^{2n-1}} +
 \frac{T^{n^2-n}}{\det(\Lambda)^n}
 \right)\\
 &+ O\left(
 \frac{T^{n^2-n-1}}{\det(\Lambda)^{n-1}} +
 \frac{T^{n^2-2}}{\det(\Lambda)^{2n-2}}
 \right).
\end{align*}

Summing over the lattices satisfying
$\alpha_{n,r}(T,\Lambda)\neq0$, we obtain
\begin{align*}
 &\sum_{\substack{\Lambda\in\mathcal{P}_{n,r}\\
 \alpha_{n,r}(T,\Lambda)\neq0}}
 \beta_{n,r}(T,\Lambda)
 \beta_{n,n-r}(T,\Lambda^\perp) =
 T^{n^2}
 \sum_{\substack{\Lambda\in\mathcal{P}_{n,r}\\
 \alpha_{n,r}(T,\Lambda)\neq0}}
 \frac{\nu_{\Lambda}^n\nu_{\Lambda^\perp}^n}
 {\det(\Lambda)^{2n}}
 +E,
\end{align*}
where, by Lemma~\ref{lem:inverse_determinant_bound},
\begin{align*}
 E
 &=
 \sum_{\substack{\Lambda\in\mathcal{P}_{n,r}\\
 \alpha_{n,r}(T,\Lambda)\neq0}}
 O\left(
 \frac{T^{n^2-1}}{\det(\Lambda)^{2n-1}} +
 \frac{T^{n^2-n}}{\det(\Lambda)^n}
 +
 \frac{T^{n^2-n-1}}{\det(\Lambda)^{n-1}} +
 \frac{T^{n^2-2}}{\det(\Lambda)^{2n-2}}
 \right)\\
 &\ll
 T^{n^2-1}
 +T^{n^2-n}\log T
 +T^{n^2-1-(n-r)}
 +\begin{cases}
     T^{n^2-2}, & \text{if } n>2,\\
     T^{n^2-2}\log T, & \text{if } n=2,
 \end{cases}\\
 &\ll T^{n^2-1}.
\end{align*}

It remains to show that the sum in the main term converges as $T \to \infty$ and to bound the truncation error arising from the condition $ \alpha_{n,r}(T,\Lambda)\neq0$.  By Lemma~\ref{lem:condition_alphanonzero_detsize}, the condition
$\alpha_{n,r}(T,\Lambda)=0$ implies that $\det(\Lambda)\ge q_nT$ for some constant $q_n>0$. Therefore, we have 
\begin{align}\label{eq:truncationerror}
    T^{n^2}\sum_{\substack{\Lambda\in\mathcal{P}_{n,r}}}
\frac{\nu_{\Lambda}^n\nu_{\Lambda^\perp}^n}
{\det(\Lambda)^{2n}} -T^{n^2}\sum_{\substack{\Lambda\in\mathcal{P}_{n,r}\\
 \alpha_{n,r}(T,\Lambda)\neq0}} 
 \frac{\nu_{\Lambda}^n\nu_{\Lambda^\perp}^n}
{\det(\Lambda)^{2n}} \ll
T^{n^2}\sum_{\substack{\Lambda\in\mathcal{P}_{n,r}\\
\det(\Lambda)\ge q_nT}} 
 \frac{1}
{\det(\Lambda)^{2n}}.
\end{align}
Dyadic decomposition as in the proof of Lemma~\ref{lem:inverse_determinant_bound} shows that 
\[
\sum_{\substack{\Lambda\in\mathcal{P}_{n,r}\\
\det(\Lambda)\ge q_nT}} 
 \frac{1}
{\det(\Lambda)^{2n}} \ll \sum_{k=\lfloor \log_2(q_nT)\rfloor}^\infty \frac{2^{(k+1)n}}{2^{2nk}} \ll T^{-n}. 
\]
Hence  the series 
\[
    c_{O_n,r} = \sum_{\substack{\Lambda\in\mathcal{P}_{n,r}}}
\frac{\nu_{\Lambda}^n\nu_{\Lambda^\perp}^n}
{\det(\Lambda)^{2n}}
\]
converges and the overall  error term in \eqref{eq:truncationerror} is  $\ll T^{n^2-n}$.
\end{proof}

\subsection{Proof of Theorem~\ref{thm:factoring_zero}}
     If $\rank(A)=0$ then  $A = O_n$ and we have $AB = O_n$ for any matrix $B \in \mathcal{M}_n(\Z)$. The contribution to $\tau_n(T,O_n)$ from such pairs $(O_n,B)$ is therefore
       \begin{align*}
    \#\left\{B\in \mathcal{M}_n(\Z): |B|\leq T \right\}
    &= 2^{n^2}T^{n^2} + O(T^{n^2-1}).
    \end{align*}
    
     If $\rank(A) = n$, then $A$ is invertible and the equation $AB = O_n$ implies that $B = O_n$. We claim that the contribution to $\tau_n(T, O_n)$ from such pairs of matrices is
    \begin{align}\label{eq:invertibleAcontribution}
    \#\left\{A\in \mathcal{M}_n(\Z): |A| \le T,~\rank(A)=n \right\}
    = 2^{n^2}T^{n^2} + O(T^{n^2-1}).
    \end{align}
To see this, it follows from the  work of Katznelson \cite{katznelson} that  the number of matrices $A$ of rank at most $n-1$
    and norm $|A|\le T$ is $\ll T^{n^2 -n}\log T$. Since the total number of matrices $|A| \le T$ is  $(2T)^{n^2}+O(T^{n^2-1})$, the claimed bound \eqref{eq:invertibleAcontribution} follows.
    
    It remains to prove an asymptotic formula for the counting function 
    \[
    \tau_n^\star(T,O_n) =\#\left\{(A,B)\in \mathcal{M}_n(\Z)^2: |A|,|B|\leq T, ~AB=O_n,~1\le\rank(A)\le n-1
\right\}.
    \]
    Let us thus assume that $r = \rank(A) \in \{ 1, \ldots, n-1\}$, and write 
    \[
    A = \begin{pmatrix}
        \textbf{ ---} & \a_1 & \textbf{--- } \\
         \textbf{ } & \vdots & \textbf{ } \\
          \text{ ---} & \a_n & \textbf{--- } 
    \end{pmatrix},  \qquad  
    B = 
    \begin{pmatrix}
        \vert & \text{ } & \vert \\
        \b_1 & \hdots & \b_n \\ 
         \vert & \text{ } & \vert 
    \end{pmatrix}.
    \]
    Then, since $AB = O_n$, we have 
    \begin{align}\label{eq:orthogonalityzero}
         \langle \a_i^t, \b_j
    \rangle = 0
    \end{align}
    for all $1 \le i,j \le n$, where $\langle \a_i^t, \b_j
    \rangle$ denotes the Euclidean inner product of the two vectors $\a_i^t$ and $\b_j$.
    There is a unique primitive lattice $\Lambda \subset \Z^n$ of $\rank(\Lambda) = r$ such that $\a_1, \ldots, \a_n \in \Lambda$, namely 
    \[
    \Lambda = \langle 
    \a_1, \ldots, \a_n \rangle_{\Q}\cap \Z^n.\]
    By \eqref{eq:orthogonalityzero}, we get $\b_1, \ldots, \b_n \in \Lambda^{\perp}$. 
    Hence we have 
    \begin{align*}
        \tau^\star_n (T, O_n) = \sum_{r=1}^{n-1} \sum_{\Lambda \in \mathcal{P}_{n,r}} 
    \alpha_{n,r} (T, \Lambda) \beta_{n,n-r} (T, \Lambda^\perp),
    \end{align*}
    where $\alpha_{n,r}(T, \Lambda)$  and $\beta_{n,n-r}(T, \Lambda^\perp)$ are defined in \eqref{eq:defalpha} and \eqref{eq:defbeta}. Now, since 
    $$\alpha_{n,r}(T,\Lambda) = \beta_{n,r}(T,\Lambda) - \gamma_{n,r}(T,\Lambda),
    $$ 
    we may put together  Lemma~\ref{lem:gamma_beta_sum_bound} and Lemma~\ref{lem:beta_beta_sum_asymp} to deduce that 
    \[
    \tau^\star_n (T, O_n) = \sum_{r=1}^{n-1} \sum_{\substack{\Lambda \in \mathcal{P}_{n,r} \\
    \alpha_{n,r}(T,\Lambda)\neq 0}} 
    \beta_{n,r} (T, \Lambda) \beta_{n,n-r} (T, \Lambda^\perp) + O(T^{n^2-1}) = c_{O_n}^\star T^{n^2} + O(T^{n^2-1}),
    \]
    where 
    \[
    c_{O_n}^\star = \sum_{r=1}^{n-1} 
     \sum_{\Lambda \in \mathcal{P}_{n,r}}
     \frac{\nu_{\Lambda}^n\nu_{\Lambda^\perp}^n}
    {\det(\Lambda)^{2n}}.
    \]

    Finally, combining all the contributions   gives  
    \[
    \tau_{n}(T, O_n) = c_{O_n}T^{n^2} + O(T^{n^2-1}),
    \]
    where 
    \begin{equation}\label{c_n}
        c_{O_n} = 2^{n^2+1} + \sum_{r=1}^{n-1} 
     \sum_{\substack{\Lambda \in \mathcal{P}_{n,r}}} 
     \frac{\nu_{\Lambda}^n\nu_{\Lambda^\perp}^n}
    {\det(\Lambda)^{2n}}
    \end{equation}
    and
    $\nu_{\Lambda}, \nu_{\Lambda^\perp}$ are defined  in \eqref{eq:villach}.
This completes the proof of the theorem.  \qed


\begin{thebibliography}{99}

\bibitem{Afif} M. Afifurrahman, Some counting questions for matrix products. {\it Bull. Aust. Math.
Soc.} {\bf 110} (2024), 32--43.


\bibitem{arala}
N. Arala, J. R. Getz, J. Hou, C.-H. Hsu, H. Li, V. Y. Wang,
A nonabelian circle method. {\it Preprint}, 2024. (\href{https://arxiv.org/abs/2407.11804}{arXiv:2407.11804})

\bibitem{BR}
G.~Bhowmik and O.~Ramar\'e,
Algebra of matrix arithmetic.
{\it J.\ Algebra} \textbf{210} (1998), 194--215.

\bibitem{bhowmik_wu_zeta2001}
G.~Bhowmik and J.~Wu,
Zeta function of subgroups of abelian groups and average orders.
{\it J.\ reine angew.\ Math.} \textbf{530} (2001), 1--15.


\bibitem{BDL}
B. J. Birch, H. Davenport and D. J. Lewis,
The addition of norm forms. {\it Mathematika} {\bf 9} (1962), 75--82.

\bibitem{chikuse_macg1990}
Y.~Chikuse, 
The matrix angular central Gaussian distribution.
{\it J.\ Multivariate Anal.} \textbf{33} (1990), 265--274.

\bibitem{davenport_on_a_principle}
H.~Davenport,
On a principle of Lipschitz.
{\it J.\ London Math.\ Soc.} \textbf{26} (1951), 179--183.


\bibitem{DRS}
W. Duke, Z. Rudnick and P. Sarnak, Density of integer points on
affine homogeneous varieties. {\it  Duke Math. J.}
{\bf 71} (1993), 143--179.

\bibitem{cartandecomp}
P. J. Forrester, 
Volumes for ${\mathrm{SL}}_N(\mathbb{R})$, the Selberg
              integral and random lattices.
{\it Found.\ Comput.\ Math.} {\bf 19} (2019), 55--82. 

\bibitem{volumesgilletgrayson}
H. Gillet and D. R. Grayson, 
Volumes of symmetric spaces via lattice points.
{\it Doc.\ Math.} {\bf 11} (2006),  425--447.

\bibitem{gorodnik_nevo}
A. Gorodnik and  A. Nevo, Counting lattice points. {\it J. reine angew. Math.} {\bf 663}
(2012), 127--176.

\bibitem{HB} 
D.R. Heath-Brown, 
A new form of the circle method, and its application to quadratic forms. 
{\it J.\ reine angew. Math.} {\bf 481} (1996), 149--206.


\bibitem{helgasongroupsandgeom}
S. Helgason, {\it Groups and geometric analysis}.
Math.\  Surveys and Monographs {\bf 83},
American Math.\ Soc., Providence, RI, 2000.


\bibitem{katznelson}
Y. R.~Katznelson,
Singular matrices and a uniform bound for congruence groups of
$\operatorname{SL}_n(\mathbb{Z})$.
{\it Duke Math.\ J.} \textbf{69} (1993), 121--136.



\bibitem{krieg} 
A. Krieg, Counting modular matrices with specified maximum norm. {\it Linear Algebra Appl.} {\bf 196} (1994), 273--278.

\bibitem{Mau}
F. Maucourant, Homogeneous asymptotic limits of {Haar} measures of semisimple linear groups and their lattices. {\it Duke Math. J.} {\bf 136} (2007), 357--399.

\bibitem{simon}
S. L. Rydin Myerson,
Quadratic forms and systems of forms in many variables.
{\it Invent. Math.} {\bf 213} (2018), 205--235.


\bibitem{Schm} W. Schmidt, Asymptotic formulae for point lattices of bounded determinant and subspaces of
bounded height.
{\it Duke Math. J.} {\bf  35} (1968), 327--339.


\bibitem{siegemvt}
C. L. Siegel, A mean value theorem in geometry of numbers.
{\it Annals of Math.} {\bf 46} (1945), 340--347.


\bibitem{siegel}
C. L. Siegel, {\it Lectures on the geometry of numbers}.
Springer-Verlag, Berlin, 1989.


\bibitem{widmer}
M.~Widmer,
Counting primitive points of bounded height.
{\it Trans.\ Amer.\ Math.\ Soc.} \textbf{362} (2010), 4793--4829.


\end{thebibliography}
\end{document}